\documentclass[11pt, reqno]{amsart}
\usepackage{amssymb}
\usepackage{amsmath}
\usepackage{amsfonts,amsthm,stmaryrd}
\usepackage{graphicx}
\usepackage[RGB]{xcolor}
\numberwithin{equation}{section} \allowdisplaybreaks
\usepackage[nobysame]{amsrefs}[11pts]

\newcommand{\beq}{\begin{equation}}
\newcommand{\eeq}{\end{equation}}

\newtheorem{thm}{Theorem}[section]

\newtheorem{lma}{Lemma}[section]

\newtheorem{remark}{Remark}[section]
\newtheorem{definition}{Definition}[section]
\newtheorem{asu}{Assumption}[section]

\usepackage{setspace}

\begin{document}
\vskip 0.2cm

\title[Singularity formation for Euler equations]{\bf Singularity formation for radially symmetric expanding wave of Compressible Euler Equations}

\author[H. Cai]{Hong Cai}
\address{Hong Cai\newline
Department of Mathematics, Qingdao University of Science and Technology, Qingdao 266061, P.R. China.}
\email{caihong19890418@163.com}

\author[G. Chen]{Geng Chen}
\address{Geng Chen\newline
Department of Mathematics, University of Kansas, Lawrence, KS 66045}
\email{gengchen@ku.edu.}

\author[T.-Y. Wang]{Tian-Yi Wang}
\address{Tian-Yi Wang \newline
	Department of Mathematics, School of Science, Wuhan University of Technology,
	Wuhan, Hubei 430070, P. R. China.}
\email{tianyiwang@whut.edu.cn; wangtianyi@amss.ac.cn}

\thanks{Corresponding author: Tian-Yi Wang}

\vskip 0.2cm

\maketitle

\begin{abstract}
In this paper, for compressible Euler equations in multiple space dimensions, we prove the break-down of classical solutions with a large class of initial data by tracking the propagation of radially symmetric expanding wave including compression. The singularity formation is corresponding to the finite time shock formation. We also provide some new global sup-norm estimates on velocity and density functions for classical solutions. The results in this paper have no restriction on the size of solutions, hence are large data results.

 \bigbreak
\noindent


\end{abstract}
%

\section{Introduction}\label{sec_1}
The compressible Euler equations in multiple space dimensions (multi-d) satisfy
\begin{equation}\label{1.1}
\begin{cases}
\rho_t+\nabla\cdot(\rho \mathbf{u})=0,\\
(\rho \mathbf{u})_t+\nabla\cdot(\rho\mathbf{u}\otimes\mathbf{u})+\nabla p=0,
\end{cases}
\end{equation}
where $\mathbf{x}\in\mathbb{R}^d,$ $d=1, 2, 3$, and $t\in \mathbb{R}^+$ are the space and time variables, respectively. The vector valued function $\mathbf{u}\in\mathbb{R}^d$ and the scalar functions  $\rho$ and $p$ stand for the velocity, density and pressure, respectively.
For the polytropic ideal gas, the pressure-density relation is
\beq\label{1.1p}
p=p(\rho)=K\rho^\gamma,\quad\hbox{with}\quad \gamma>1,
\eeq
where $K$ is a positive constant, and $\gamma$ is the  adiabatic constant.

In this paper, we first study global behaviors of classical solutions for the initial value problem of \eqref{1.1} with radial symmetry. Then, we prove the singularity formation for solutions with a large class of initial data involving initial compression. Here the radially symmetric solution of  \eqref{1.1} satisfies the following symmetric transformation:
\begin{equation*}
\rho(\mathbf{x},t)=\rho(r,t),\quad \mathbf{u}(\mathbf{x},t)=u(r,t)\frac{\mathbf{x}}{r}, \quad \hbox{with radius}\quad r=|\mathbf{x}|.
\end{equation*}
Then the functions $(\rho,u)$ are governed by the following Euler equations
\begin{equation}\label{1.2}
\begin{cases}
(A\rho)_t+(A\rho u)_r=0,\\[2mm]
(A\rho u)_t+(A\rho u^2)_r+Ap_r=0,
\end{cases}
\hbox{with}\quad A(r)=r^{d-1},\quad d=1, 2, 3.
\end{equation}
We consider the initial value problem of system \eqref{1.2} for $(r,t)\in[0,\infty)\times[0,\infty)$, subject to the following initial conditions:
\begin{equation}\label{inbo}
\rho(r,0)=\rho_0(r),\qquad u(r,0)=u_0(r).
\end{equation}
%
When $r>0$, system \eqref{1.2} can be written as
\begin{equation}\label{1.2+}
\begin{cases}
\rho_t+(\rho u)_r=-\frac{A'}{A}\rho u,\\[2mm]
(\rho u)_t+(\rho u^2+p)_r=-\frac{A'}{A}\rho u^2.
\end{cases}\end{equation}
See \cite{Dafermos2010} for the derivation of \eqref{1.2} or \eqref{1.2+}.


\smallskip

 The compressible Euler system is the most fundamental model for system of hyperbolic conservation laws.
It is well known that, due to nonlinearity,
classical solutions of systems of hyperbolic conservation laws, may form gradient blowup in finite time, even when initial data are smooth. 
This is a physical effect manifested by the development of shock waves, at which the conserved variables become discontinuous.

The study of breakdown of classical solutions for hyperbolic conservation laws has a long history \cite{Dafermos2010}. One can trace back to Stokes in \cite{stokes} for a breakdown example for some scalar
equation. 
In one space dimensional (1-d), classical solutions may break down, even under small initial oscillation. The result was initiated by the pioneer work of Lax in  \cite{lax2} in 1964 for strictly hyperbolic genuine nonlinear systems with two unknowns, then followed by John, Li-Zhou-Kong and Liu in \cite{Fritzjohn, Li daqian, liu}, and \textit{etc.} for more general systems of conservation laws.

For large data 1-d problem of \eqref{1.1}. i.e. \eqref{1.2} with $d=1$, Lax' method in \cite{lax2} can directly provide an equivalent condition on the initial data for singularity formation when $\gamma\geq 3$. The result can be read as: The finite time singularity, due to shock wave, forms if and only if the initial data include compression.

However, when $1<\gamma<3$, to prove the same equivalent condition on singularity formation,
one needs to first prove a proper density lower bound estimate, since the Riccati equations will degenerate while density approaches zero. See examples with density approaching zero as $t\rightarrow \infty$ in \cite{courant,lin2, jenssen}.  Such a time-dependent lower bound was first found by Chen-Pan-Zhu in \cite{CPZ} then improved to its optimal order $O(1/t)$ by Chen in \cite{G9}. Furthermore, singularity formation for 1-d solutions of non-isentropic Euler equations was proved when the initial compression is stronger than some threshold \cite{CPZ}.

\smallskip
For 3-d solutions of compressible Euler equations, Sideris proved some singularity formation results when considering the initial locally supersonic data, in a pioneer work \cite{Sideris}. Based on
a contradiction argument, he proved that the corresponding solutions cease to be $C^1$ in finite time.
The solution could be rotational with non-trivial (dynamic) entropy. Unfortunately this approach fails to provide any information on the nature of breakdown or identify the actual time of blow-up. Similar results can be found in \cite{Rammaha, Sideris1} for the 2-d case and  in \cite{Makino, Leiduzhang} for the radially symmetric case.

For solutions with small perturbation, Alinhac proved a blowup result for 2-d barotropic compressible Euler equations with radial symmetry \cite{Alinhac1}. A sharp estimate on the life span of classical solution was also given.
Later, Alinhac \cite{Alinhac-book1, Alinhac2, Alinhac3} proved the blowup of classical solutions for a large class of 2-d quasilinear wave equations that fail to satisfy the null condition in \cite{Klainerman}. The isentropic irrotational compressible Euler equations, linked to the quasilinear wave equations, are included.

Another method to study the singularity formation of multi-d solutions is to use the geometric framework introduced by Christodoulou \cite{Chris}. 
He first studied wave equations modeling isentropic, irrotational special relativistic fluid mechanics, and found a shock formation example with some compactly supported initial
data near a given constant state. Later,  Christodoulou-Miao \cite{shuang} studied shock formation for small and compactly supported perturbations of constant solutions to the non-relativistic compressible Euler equations. In particular, for isentropic, irrotational initial data, the work \cite{shuang} yielded a precise description on the singularity formation detected by Sideris \cite{Sideris}, and revealed the fine geometric property of the singular hyperplane (shock front). Recently, Luk and Speck \cite{Speck} considered the 2-d case with small but non-zero vorticity, even at the location of the shock. This result provides the description of the vorticity near a singularity formed from compression.

\smallskip
In this paper, we focus on the singularity formation for radially symmetric multi-dimensional solutions of Euler equations satisfying \eqref{1.2}. 
For radially symmetric solutions, one can still track the propagation of any initial compression wave along characteristic using some Riccati system on gradient variables. Here we use the system established by Chen-Young-Zhang in \cite{G8}. 
However, the geometric effect at the origin and the inhomogeneity caused by the varying $A(r)$ make the problem very different from the 1-d problem.

One of the most challenging parts is to establish the uniform upper bounds on $|u|$ and $\rho$ and the time-dependent lower bound on $\rho$. These global bounds shall hold for classical solutions in the whole space $\{(r,t)\,|\,r\geq0,\ t\geq0\}$ including the origin.

The key ideas, used in this paper, to establish the uniform bounds in the whole space, come from finding invariant regions on both planes of  $(u, \rho)$ and some gradient variables. Here the invariant region on the $(u, \rho)$ plane found in this paper lies in the half plane $u\geq0$. This means that solutions considered in this paper only include  expanding waves, traveling away from the origin. Applying the Riccati system
in Chen-Young-Zhang \cite{G8}, one can track the propagation of strong compression wave then show the finite time blowup. On the other hand, we find global $L^\infty$ bounds on velocity and density before singularity formation. So it is reasonable to believe that the singularity formation is due to the shock formation, where we refer the readers to some earlier works in \cite{Chenshuxing,Chenshuxing2,Kongdexing}. Since there needs no restriction on the size of solutions, all results in this paper are large data results.


Another motivation of this paper, besides studying the shock formation, is to provide useful $C^1$ and $L^\infty$ estimates on velocity and density for classical solutions. These estimates are very useful for the future construction of interesting finite time or global-in-time solutions for \eqref{1.2}.
In fact, for 1-d solutions, a method to construct global special solutions using $C^1$ and $L^\infty$ estimates on velocity and density was provided in a very recent paper in \cite{CCZ}. Currently, for radially symmetric solutions,  we are still lack of a framework on the global BV existence of solutions in the whole space. For the isothermal gas, see BV existence results for the {\em exterior} problem of flow outside of a fixed ball with
reflecting boundary conditions in \cites{mmu,mmu2,mt}. The $L^\infty$ existence for a lager class of initial data for the isentropic gas in the whole space was established in \cite{Chen-1}, while the $L^p$ case are given in \cite{lfw,cp1}, by using the compensated compactness method.


In this paper, for a large class of supersonic initial data, we construct the local existence and uniform upper bound along with the time depend density lower bound, without smallness restriction. By transfer the spacial gradient to new variables satisfying Riccati type equations, the gradient blowup is shown, which indicate the formation of shock.

This paper is divided into six sections. We will introduce our main results in Section 2. In Sections 3 and 4, we will show the local existence and a uniform upper bound estimate. In Section 5, we will prove the results on low bound of density. Finally, in Section 6, we will prove the gradient blowup theorem.

\section{Main results}
In this section, we introduce the main results of this paper. To begin with, we give the definition of classical solutions to system \eqref{1.2}--\eqref{inbo}.
\begin{definition}\label{def_cl}
Let $T$ be any given positive constant. The functions $(\rho,u)(r,t)$ are called a classical solution to the initial value problem of compressible Euler equations \eqref{1.2}--\eqref{inbo} for any $t\in[0,T)$, if
\[\rho\in C^1(\mathbb{R}^+\times[0,T)),~~u\in C^1(\mathbb{R}^+\times[0,T)),\]
solve the system \eqref{1.2} for any $(r,t)\in\mathbb{R}^+\times[0,T)$,
and satisfy the initial condition \eqref{inbo} point-wisely on $r\in\mathbb{R}^+$.
Here $\rho$ is non-negative.
\end{definition}

Our first result is a local existence theorem, which provides a basis on the singularity formation result, \textit{i.e.} the singularity forms from a local-in-time $C^1$ solution. In the same time, we also provide the
uniform $L^\infty$ upper bound on $|u| $ and $\rho$.
First, we introduce the following initial assumption.
\begin{asu}\label{asu_1}
Assume the initial data $(\rho_0,u_0)(r)\in (C^1([0,\infty)))^2$,
$(\frac{\rho_0^{\frac{\gamma-1}{2}}}{r}, \frac{u_0}{r})$ $\in (C^1[0,1))^2$ and there exists a uniform constant $C_0$ such that
\begin{equation}\label{inbo0}
0<\frac{2\sqrt{K\gamma}}{\gamma-1}\rho_0^{\frac{\gamma-1}{2}}(r)\leq u_0(r)\leq C_0,
\end{equation}
for any $r\in(0,\infty)$, and
\begin{equation}
\label{inbo2}
\rho_0(0)=u_0(0)=0.
\end{equation}
\end{asu}

The local existence theorem can be stated as following. We also achieve some uniform $L^\infty$ upper bounds on $|u|$ and $\rho$, as long as the solution is still in $C^1$.
\begin{thm}\label{thm_bounds}
For $1<\gamma\leq3$, when the initial value problem \eqref{1.2}--\eqref{inbo} satisfies the  initial Assumption \ref{asu_1}, there exists $\delta>0$ and a unique solution $(\rho, u)\in (C^1([0, +\infty)\times(0, \delta)))^2$, with
\begin{equation}\label{inbo3}
0< \frac{2\sqrt{K\gamma}}{\gamma-1}\rho^{\frac{\gamma-1}{2}}(r,t)\leq u(r,t)\leq 2C_0,
\end{equation}
for any $r\in(0,\infty)$ and $\rho(0,t)=u(0,t)=0$. Furthermore, for any $\bar{t}>0$, $\bar{r}>0$, the backward characteristics will not touch the origin $r=0$.
\end{thm}


In this proof, the main idea is to observe an invariant domain on $t$, when $r\in(0,\infty)$:
\[
 \left\{(w,z)\, \big|\, 0\leq z< w\leq2C_0\right\}
\] 
where
\[
w:=u+\frac{2\sqrt{K\gamma}}{\gamma-1}\rho^{\frac{\gamma-1}{2}},\quad z:=u-\frac{2\sqrt{K\gamma}}{\gamma-1}\rho^{\frac{\gamma-1}{2}},
\]
are two Riemann variables. Here $C_0$ is the arbitrary constant in the initial Assumption \eqref{asu_1}.

From the physic point of view, 
\begin{equation}
\frac{2\sqrt{K\gamma}}{\gamma-1}\rho^{\frac{\gamma-1}{2}}(r,t)\leq u(r,t)
\end{equation}
equals to $Ma>\frac{2}{\gamma-1}$, while $Ma$ is the Mach number. \eqref{inbo0} and \eqref{inbo3} shows if the flows are sufficient supersonic $Ma>\frac{2}{\gamma-1}\geq1$, the flows will keep supersonic. The similar property is exploded in \cite{Chen-1} in the study of $L^\infty$ weak entropy solution.

Next, we will provide a lower bound estimate on density in the domain $(x,t)\in [b,+\infty)\times [0,\infty)$,
for any $b>0$. Note one cannot achieve any positive lower bound on density in the whole plane since the density is always zero at the origin.

The proof relies on finding an invariant domain on some gradient variables in the Lagrangian coordinates. This idea has been first used in \cite{G9} for 1-d solutions. We will state the main results (Theorems \ref{thm_den1}-\ref{thm_den2}) later, after introducing the Euler equations \eqref{La} in the Lagrangian coordinates.

Finally, we present the main theorem on singularity formation. To state the theorem, we first introduce
the Lagrangian coordinates $(x',t')$
\[
x'=\int_0^r A(r)\rho(r)\,dr, \quad t'=t,
\]
where $A$ is defined in \eqref{1.2}. For any time, $x'\in[0,\infty)$ is a strictly increasing function on $r$.
\begin{thm}\label{thm_2.4}
For $1<\gamma<3$ and $d=2,3$, assume the initial data satisfy conditions  in Assumption \ref{asu_1} and $(\alpha(x',0)$, $\beta(x',0)$, $\tilde{\alpha}(x',0)$, $\tilde{\beta}(x',0))$ defined in \eqref{ab1} and \eqref{ab2} are all uniformly bounded.  Assume one of the following two conditions holds:

(i) There exist some $\bar x'$ and $T$: satisfying condition \eqref{Test}, such that
\begin{equation*}
Y(\bar x',0)<-G_{\gamma,T}(\bar x'), \quad\hbox{when}\quad \gamma\neq\frac{5}{3};
\end{equation*}
or satisfying condition \eqref{Test1}, such that
\begin{equation*}
Y_1(\bar x',0)<-G_{\gamma,T}(\bar x'),\quad\hbox{when}\quad \gamma=\frac{5}{3}.
\end{equation*}

(ii)  There exist some $\bar x'_0>\bar x'>0$ and $T$: satisfying condition \eqref{Test} and $r(\bar{x}'_0,0)-r(\bar{x}',0)\geq 2C_0\gamma T$, such that
\begin{equation*}
Q(\bar x'_0,0)<-G_{\gamma,T}(\bar x') ,\quad\hbox{when}\quad \gamma\neq\frac{5}{3};
\end{equation*}
or satisfying condition \eqref{Test1} and $r(\bar{x}'_0,0)-r(\bar{x}',0)\geq 2C_0\gamma T$, such that
\begin{equation*}
\quad Q_1(\bar x'_0,0)<-G_{\gamma,T}(\bar x') ,\quad\hbox{when}\quad \gamma=\frac{5}{3}.
\end{equation*}

Then the $C^1$ solutions of \eqref{La} will break down in finite time before $t=T$. Here $G_{\gamma,T}(x')>0$ will be defined in \eqref{GT} and \eqref{YQE1} for different $\gamma$ values.  The gradient variables $Y$, $Q$, $Y_1$ and $Q_2$ will be defined in \eqref{YQ} and \eqref{YQ1}.
\end{thm}
When $\gamma=3$, we have a stronger version of singularity formation.
\begin{thm}\label{thm_main3}
For $\gamma=3$, assume the initial data satisfy conditions in Assumption \ref{asu_1}. If there exists some $\bar{x}'>0$ such that,
\begin{equation}\label{Hcritical}
Y_2(\bar{x}',0)<-H(\bar{x}'), \quad\hbox{or}\quad Q_2(\bar{x}',0)<-H(\bar{x}'),
\end{equation}
with $H(\bar{x}')$ defined in \eqref{YQE2},
 then the $C^1$ solution of \eqref{La} will break down in finite time. Here $Y_2$, $Q_2$ are defined in \eqref{YQ2}.
\end{thm}

Since when $d=1$, an if and only if condition on the initial date for singularity formation was proved in \cite{lax2,CPZ}, we do not include the 1-d case in this theorem.

We note that if $Y$ and $Q$ (or $Y_1,Q_1$, $Y_2,Q_2$ for other $\gamma$ value) are both bounded away from negative infinity when $t\in[0,T)$, one can prove the global existence up to time $T$, using the uniform $L^\infty$ estimates on $\rho$ and $u$. The upper bounds on $Y$'s and $Q$'s can be proved using the results in Theorems \ref{thm_den1}-\ref{thm_den2}, where the case when $\gamma=3$ can be treated in a similar method. Here $T$ can be infinity.

By the a priori $L^\infty$ estimate in \eqref{thm_bounds} and the time dependent density lower bound, we know when singularity forms, the state variables $u$, $\rho$ are both finite, and the
specific volume $1/\rho$ is finite except on the line $r=0$.
On the other hand, $C^1$ solution exists before the earliest time of gradient blowup calculated by the Riccati system.
So it is reasonable to believe that the singularity forms due to the shock formation. We refer the readers to some earlier works \cite{Chenshuxing, Chenshuxing2, Kongdexing} showing why the singularity is a shock in the $1$-$d$ case.

Finally we give a remark on the initial assumption \eqref{inbo2}.
\begin{remark}
The assumption $u_0(0)=0$ is a reasonable physical assumption, under which the system \eqref{1.2} is satisfied at the origin.

To make the initial assumption \eqref{inbo3} satisfied on the entire half line $0\leq r<\infty$,
we must assume that $\rho_0(0)=0$. Then we can obtain a global $L^\infty$ estimate on density and velocity using the invariant domain argument, before singularity formation.

On the other hand, we can still claim a singularity formation result when $\rho_0(0)>0$. In fact, we can consider the initial data with $\rho_0(0)>0$, but satisfying \eqref{inbo0} on a set $r>\varepsilon$ for some $\varepsilon>0$. By restricting the consideration on the domain of dependence corresponding to the initial ray $r>\varepsilon$ at $t=0$, we can still establish the $L^\infty$ bounds similarly as in Theorem \ref{thm_bounds} and Theorems \ref{thm_den1} and \ref{thm_den2}, then derive a finite time singularity formation result. Here we omit the detail of this construction. But we note that for such a result, we do not have any global $L^\infty$ estimates including the origin, before blow-up, using the methods in the current paper.

If the initial assumption is satisfied, no matter on the whole half line $r\geq 0$ or on $r>\varepsilon$, the singularity formation considered in this paper is due to the concentration of compression wave propagating in the outward direction in the Eulerian coordinates, i.e. both wave speeds $c_1,\ c_2\geq 0$ when $1<\gamma\leq 3$, by \eqref{cdef} and \eqref{inbo3}.
\end{remark}

\section{Local existence}\label{sec_3}
Let's first introduce some notations. We define the function $h$ as
\begin{equation*}\label{h_def}
h:=\sqrt{K\gamma}\rho^{\frac{\gamma-1}{2}},
\end{equation*}
which takes the role of $\rho$.
Then we define the Riemann variables as
\begin{equation*}\label{swdef}
\textstyle
w:=u+\frac{2}{\gamma-1}h,\quad z:=u-\frac{2}{\gamma-1}h.
\end{equation*}

For any smooth solution away from the center, the equation \eqref{1.2+} could be diagonalized as:
\begin{equation}\label{sw}
\begin{cases}
w_t+c_2w_r=-\frac{d-1}{r}uh,\\
z_t+c_1z_r=\frac{d-1}{r}uh,
\end{cases}
\end{equation}
where two characteristic speeds are denoted as
\begin{equation}\label{cdef}
\begin{cases}
c_1=u-\sqrt{p'(\rho)}=u-h,\\
c_2=u+\sqrt{p'(\rho)}=u+h.
\end{cases}
\end{equation}

Instead of  considering a boundary value problem in the region $r\geq 0$, we can consider an initial value problem with initial data given on $r\in \mathbb R$. Here we
add the solution when $r<0$ by reflecting the solution with positive radius.

Assume that, when $r<0$,
\[
u(r,0)=-u(-r,0), \quad h(r,0)=h(-r,0).
\]
Then it is easy to see that, when $r<0$,
\[
c_1(r,0)=-c_2(-r,0),\qquad
c_2(r,0)=-c_1(-r,0),
\]
and
\[
w(r,0)=-z(-r,0),\qquad
z(r,0)=-w(-r,0).
\]
So, if $(u(r,t), h(r,t))$ is a solution for \eqref{sw} on
$(r,t)\in[0,\infty)\times[0, \delta)$ with $u(0,t)=h(0,t)=0$, then
 $(u(r,t), h(r,t))$ is a solution for \eqref{sw} on
$(r,t)\in(-\infty,\infty)\times[0, \delta)$.

Next to cope with the vacuum at the origin, we need to introduce the weighted variables:
\[\tilde{u}:=\frac{u }{r},\qquad \tilde{h} =\frac{h }{r},\qquad \hbox{for any } r\in\mathbb R,\]
and
\[\tilde{w} =\tilde{u} +\frac{2}{\gamma-1}\tilde{h},\qquad \tilde{z} =\tilde{u} -\frac{2}{\gamma-1}\tilde{h} .\]
Then \eqref{sw} is reformulated as:
\begin{equation}\label{stwh}
\begin{cases}
\tilde{w} _t+c_2  \tilde{w} _r=-(d-1)\tilde{u}  \tilde{h} -(\tilde{u} +\tilde{h} )\tilde{w} ,\\
\tilde{z} _t+c_1  \tilde{z} _r=(d-1)\tilde{u}  \tilde{h} -(\tilde{u}-\tilde{h} )\tilde{z} .
\end{cases}
\end{equation}

Since we assume that $\|\tilde{h}_0,\tilde{u}_0\|_{C^1}$ is finite. So there exists $T_\delta>0$, such that the solution of \eqref{stwh} exists in $t\in[0,T_\delta]$. And the life-span $T_\delta<\infty$ if and only if $\|\tilde{h},\tilde{u}\|_{C^1}$ blows up in finite time. This result is standard, since \eqref{stwh} includes no singular term. The reader can find a proof of local existence at Theorem 2.1 in \cite{Li-book2}, and a global one when $\|\tilde{h},\tilde{u}\|_{C^1}$ is a priorily bounded in \cite{Lidaqian}.

Now we transform the $C^1$ solution of \eqref{stwh} to \eqref{sw} when $r>0$ under Assumption \ref{asu_1} on the initial data.

First, we claim that $\tilde h\neq 0$ for any $r\neq 0$.
In fact, by \eqref{stwh}, we know that
\beq\label{local_est}
{\tilde h}_t+r {\tilde u} {\tilde h}_r= -\left({\tilde u}+\frac{\gamma-1}{2} (r\tilde u)_r +\frac{\gamma-1}{2}(d-1){\tilde u}\right){\tilde h}.
\eeq
So along any flow map
\[
\frac{dr(t)}{dt}=r {\tilde u},
\]
for any $C^1$ solution of \eqref{stwh}, we have
\[
\tilde h(r(t),t)\neq 0, \quad\hbox{if}\quad \tilde h(r(0),0)\neq 0.
\]

By Assumption \ref{asu_1}, $\tilde h(r,0)$ possibly vanishes only when $r=0$, where the corresponding flow map is $r(t)\equiv 0$. Hence,
\[\tilde h(r,t)\neq 0,\quad\hbox{when}\quad  r\neq 0.\]
As a consequence, also by \eqref{local_est},
\[
h(r,t)> 0,\ \rho(r,t)>0 \quad\hbox{when}\quad  r\neq 0.
\]

Since $\tilde{h},\tilde{u}\in {C^1}$, we know $h(0,t)=u(0,t)=0$ and $c_1(0,t)=c_2(0,t)=0$. So
 the characteristic $r=0$ will split the positive radius and negative radius regions.
When $r\in(0,\infty)$, the $C^1$ solution of \eqref{stwh} will give a non-vacuum solution of \eqref{sw}, under Assumption \ref{asu_1}. Adding $h(0,t)=u(0,t)=0$, we get a local solution of \eqref{sw}, before any gradient blowup.

Note, the method in this section will not give a useful density lower bound for a future singularity formation result, although we can prove there is no vacuum when $r>0$. A better density lower bound will be given in section \ref{sec_4}.

Now, we want to show: for any $\bar{t}>0$, $\bar{r}>0$, the backward characteristics will not touch the origin $r=0$. First, the backward characteristics are defined as follow:
\begin{equation*}
l_1:\quad \frac{d r_1}{d t}=c_1 (r_1, t)=u (r_1, t)-h (r_1, t),
\end{equation*}
and
\begin{equation*}
l_2:\quad \frac{d r_2}{d t}=c_2 (r_2, t)=u (r_2, t)+h (r_2, t),
\end{equation*}
with $r_1(\bar{t})=r_2(\bar{t}=\bar{r}$. From above properties, before touch the origin $r=0$, $\rho>0$ which implies $r_1(t)<r_2(t)$ for $t<\bar{t}$. So, one just need to show $r_1(t)>0$ for  $t<\bar{t}$. Due to $h(0,t)=u(0,t)=0$,
\begin{equation}
c_1(r, t)=c_1(r,t)-c_1(0,t)\leq L r,
\end{equation}
while $L$ is Lipschitz constant of $h$ and $u$ . Then, for $l_1$, 
\begin{equation}
\frac{d r_1}{d t}\leq L r_1,
\end{equation}
which leads to for any $t<\bar{t}$, 
\begin{equation}
r_1(t)>e^{L(t-\bar{t})}\bar{r}>0.
\end{equation}
Hence, the both backward characteristics will not touch the origin. And this porosity will still hold after smooth transformation, like from Eulerian to Lagrangian coordinates.

\section{An invariant domain on density and velocity: upper bound}\label{sec_bou}

In this section, we give a uniform upper bound on density and velocity, by finding an invariant domain.

\begin{lma}
Suppose $(h, u)$ is any $C^1$ solution of \eqref{sw} when $r>0$ and  $0<t<T$, satisfying
\begin{equation*}
0\leq u_0-\frac{2}{\gamma-1}h_0 < u_0+\frac{2}{\gamma-1}h_0\leq 2C_0,
\end{equation*}
and other initial conditions in Assumption \ref{asu_1}, then
\begin{equation}\label{uni-bound-h}
0\leq u (r, t)-\frac{2}{\gamma-1}h (r, t)< u (r, t)+\frac{2}{\gamma-1}h (r, t)\leq 2C_0,
\end{equation}
i.e.
\[
0\leq z< w\leq 2C_0,
\]
when $r>0$ and  $0<t<T$.
\end{lma}

\begin{remark}
The above result implies:
\begin{equation*}
0<\frac{2}{\gamma-1}h (r, t)\leq u (r, t)\leq 2C_0,
\end{equation*}
when $r>0$.
\end{remark}

\begin{proof}

%

We first prove $z\geq0$ by employing a contradiction argument.
Assume $z (\bar{r},\bar{t})<0$  at some point $(\bar{r},\bar{t})$ with $\bar{t}>0$. 
We still use $l_1$ to denote the 1-characteristic $r=r_1(t)$ though the point $(\bar{r},\bar{t})$:
\[
\frac{d r_1}{dt}=c _1(r_1(t), t),\qquad r_1(\bar{t})=\bar{r}.\]
See Figure \ref{Fig:zest}.
\begin{figure}[h!]
\centering
 \scalebox{0.5}{\includegraphics{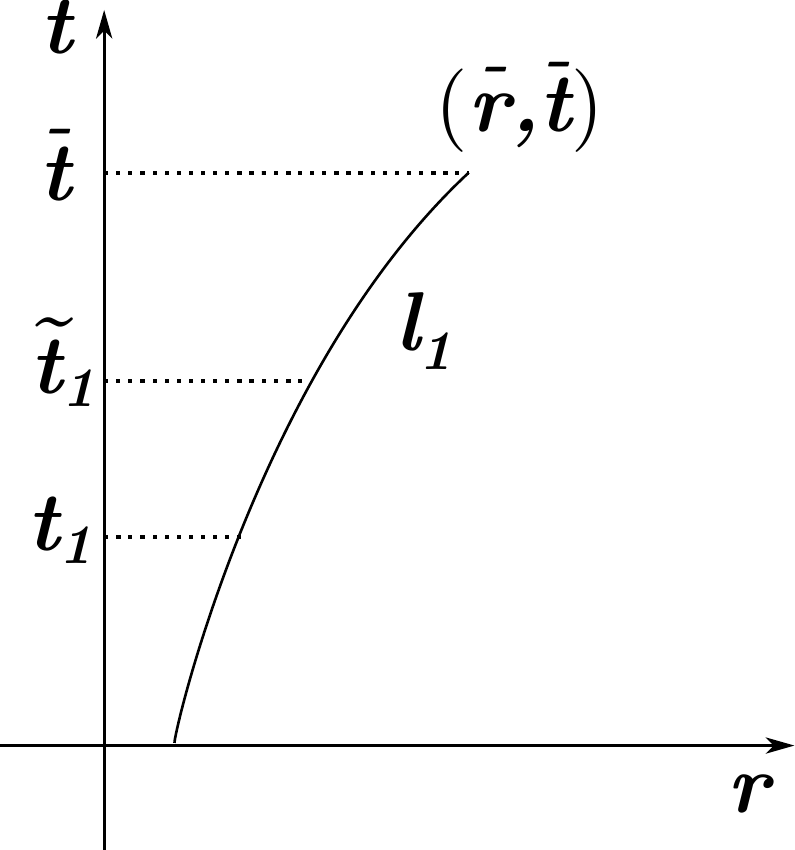}}
\caption{Proof of Lemma 4.1.}\label{Fig:zest}
\normalsize
\end{figure}

Then we can find a time $t_1$ with $0<t_1<\tilde{t}_1\leq \bar{t}$, such that, along the characteristic $r_1(t)$, 
\begin{itemize}
\item $z(r_1(t), t)\geq0$ when $t\in[0, t_1)$, 
\item $z(r_1(t_1), t_1)=0$, and 
\item $z(r_1(t), t)<0$ when $t\in(t_1,\tilde{t}_1]$.
\end{itemize}
Since $h >0$ along $r_1$, now we know that $0<\frac{2}{\gamma-1}h \leq u $ when $t\in[0, t_1]$, hence,
\begin{equation*}
z _t+c_1  z _r=\frac{d-1}{r}u  h >0.
\end{equation*}
This contradicts to that $z({r}_1(t_1),t_1)=0$ while $z({r}_1(t),t)<0$ when $t_1< t\leq \tilde{t}_1$. Therefore $0\leq z $.

Then clearly for $t>0$, $r>0$, $u(r, t) > \frac{2}{\gamma-1}h (r, t)>0$ and
	\begin{equation*}
	w _t+c _2 w _r=-\frac{d-1}{r} u  h  <0,
	\end{equation*}
	which implies $w (t)\leq w(0)$ along any 2-characteristic $\frac{dr_2}{dt}=c_2$.
	Therefore, we have
	$$
	0\leq z < w \leq 2 C_0.
	$$
\end{proof}


\section{Lower bound of density}\label{sec_4}
The aim of this section is to give some lower bound estimates on the density function in classical solutions,
when $1<\gamma\leq3$ for most physical cases. To find these estimates, it is much more convenient to study
the solution in the Lagrangian coordinates. For convenience, we always assume that $d=2,\, 3$ in this section.
In fact, when $d=1$, there is an $O(1+t)^{-1}$ estimate on lower bound of density in \cite{G9}.
\subsection{Lagrangian coordinates}
In this section, we use Lagrangian coordinates $(x',t')$ transformed from the Eulerian coordinate $(r,t)$, defined by
\begin{equation}\label{Laco}
x'=\int_0^r A(r)\rho(r)\,dr, \quad t'=t,
\end{equation}
where
\begin{equation*}\label{Adef}
A:=A(r)=r^{d-1}.
\end{equation*}

For $C^1$ solutions satisfying Assumption \ref{asu_1} on the initial data,  $\rho(r,t)>0$ when $r>0$, while $\rho(0,t)=0$, by Theorem \ref{thm_bounds}.
So we know
$x'(r,t)$ is well-defined and strictly increasing on $r$ for any $t\geq 0$, with $x'(0,t)=0$.
And
$r(x',t')$ is well-defined and strictly increasing on $x'$ for any $t'$.

Using system (\ref{1.2}), it is easy to check that the transformation satisfies
\begin{equation*}\label{lt1}
\left\{
\begin{split}
dx'&=A\rho dr-A\rho u dt\,,\\[2mm]
dt'&=dt,
\end{split}
\right.
\end{equation*}
and
\begin{equation}\label{lt2}
\left\{
\begin{split}
\frac{\partial}{\partial t}&=-A\rho u \frac{\partial}{\partial x'}+\frac{\partial}{\partial t'}\,,\\[2mm]
\frac{\partial}{\partial r}&=A\rho \frac{\partial}{\partial x'}\,.
\end{split}
\right.
\end{equation}

We define a new variable
\begin{equation*}\label{vdef}v:=\frac{1}{A(r)\rho},\end{equation*}
 then system \eqref{1.2} can be written in the Lagrangian frame as
\begin{equation}\label{La}
\left\{
\begin{split}
v_{t'}-u_{x'}=0,\\[2mm]
u_{t'}+Ap_{x'}=0.
\end{split}
\right.
\end{equation}
For convenience, we introduce a new variable
\beq\label{eta_def}\eta:=\frac{2\sqrt{K\gamma}}{\gamma-1}A^{\frac{\gamma-1}{2}}\rho^{\frac{\gamma-1}{2}}
\eeq
to take the place of $\rho$. So the state variables $v$ and $p$, which satisfies \eqref{1.1p}, can be written as
\begin{equation*}
v=K_\tau\eta^{-\frac{2}{\gamma-1}}\quad {\rm and}\quad p=K_pA^{-\gamma}\eta^{\frac{2\gamma}{\gamma-1}},
\end{equation*}
and the (Lagrangian) wave speed is
\[C:=\sqrt{-Ap_v}=\sqrt{K\gamma}\,A^{-{\frac{\gamma-1}{2}}}v^{-\frac{\gamma+1}{2}}=K_cA^{-\frac{\gamma-1}{2} }\eta^{\frac{\gamma+1}{\gamma-1}},\]
where $K_\tau, K_p$ and $K_c$ are positive constants given by
\begin{equation*}\label{Ktpc}
K_\tau:=(\frac{2\sqrt{K\gamma}}{\gamma-1})^\frac{2}{\gamma-1},\quad K_p:=KK_\tau^{-\gamma}\quad {\rm and }\quad K_c:=\sqrt{K\gamma}K_\tau^{-\frac{\gamma+1}{2}}.\end{equation*}
The forward and backward characteristics are described by
\[\frac{dx'}{dt'}=C\quad \hbox{and}\quad {\rm}\quad \frac{dx'}{dt'}=-C,\]
respectively. We denote the corresponding directional derivatives along these characteristics by
\[\partial_+:=\frac{\partial}{\partial t'}+C\frac{\partial}{\partial x'}\quad{\rm and }\quad\partial_-:=\frac{\partial}{\partial t'}-C\frac{\partial}{\partial x'}.\]
By \eqref{lt2}, it is easy to check that these derivatives are equivalent to derivatives along $1$ and $2$ characteristics  in the Eulerian coordinates, respectively,
\[
\partial_+=\frac{\partial}{\partial t}+c_2\frac{\partial}{\partial r}\quad\hbox{and}\quad
\partial_-=\frac{\partial}{\partial t}+c_1\frac{\partial}{\partial r}.
\]
We notice that along any 1 or 2-characteristic, radius $r$ is not decreasing on time since $c_1$ and $c_2$ are always nonnegative by \eqref{inbo3} when $1<\gamma\leq 3$.

In these coordinates, classical solutions of \eqref{La} satisfy the system
\begin{equation}\label{3.2}
\left\{
\begin{split}
\eta_{t'}+CA^{\frac{\gamma-1}{2}}u_{x'}=0,\\[2mm]
u_{t'}+CA^{-\frac{\gamma-1}{2}}\eta_{x'}-\gamma pA_{x'}=0,
\end{split}
\right.
\end{equation}
by \eqref{1.2+}.

\subsection{Riccati equations and key ideas}
The crucial point in \cite{G9}, to obtain the lower bound estimate on density for 1-d solutions, is to find an uniform upper bound on some gradient variables measuring compression and rarefaction.

For isentropic solutions in multiple space dimensions, what is the best way to define the rarefaction and compression, is not as clear as for solutions in one space dimension where one can use the derivate of Rieman invariants, such as in \cite{lax2,G3,G6}.

One natural good choice of gradient variables is the following one, extending from the derivatives of Riemann invariants in one space dimension, which was first given in \cite{G8},
\begin{equation}\label{ab1}
\begin{split}
&\alpha:=u_{x'}+A^{-\frac{\gamma-1}{2}}\eta_{x'},\\[2mm]
&\beta:=u_{x'}-A^{-\frac{\gamma-1}{2}}\eta_{x'}.
\end{split}
\end{equation}
These variables satisfy the following coupled Riccati equations, where the detail calculations can be found in \cite{G8},
\beq\label{Ri}
\left\{
\begin{split}
\partial_+\alpha=k_1(\alpha\beta-\alpha^2)+k_2^+(\alpha-\beta)+F,\\[2mm]
\partial_-\beta=k_1(\alpha\beta-\beta^2)+k_2^-(\beta-\alpha)+F,
\end{split}
\right.
\eeq
where
\begin{equation*}
\begin{split}
&k_1=\frac{\gamma+1}{2(\gamma-1)}K_c\eta^{\frac{2}{\gamma-1}}>0,\\
&k_2^\pm=-\frac{\gamma-1}{4}uA^{-1}\dot{A}\pm\frac{3(\gamma-1)^2}{8}\eta A^{-\frac{\gamma+1}{2}}\dot{A},\\
&F=\frac{(\gamma-1)^3}{8K_c}\eta^{\frac{2\gamma-4}{\gamma-1}}A^{-\gamma-1}(A\ddot{A}-\gamma\dot{A}^2).\\
\end{split}
\end{equation*}
Here
\[\dot{A}:=\frac{dA(r)}{dr}, \quad{\rm and}\quad \ddot{A}:=\frac{d^2A(r)}{dr^2}.\]
 In the coordinate $(r,t)$, it is easy to check that
 \begin{equation*}\label{F}
 A\ddot{A}-\gamma\dot{A}^2<0\quad\hbox{when}\quad d=2,3.\end{equation*}
 So it is easy to see that
 \beq\label{F_in}
 F<0\quad\hbox{when}\quad x'> 0.
 \eeq

Now we restrict our consideration on the case when Assumption \ref{asu_1} on the initial data is satisfied.  By Theorem \ref{thm_bounds}, we know that for any classical solutions in $(x',t')\in (0,+\infty)\times[0,T)$ satisfying initial conditions in Assumption \ref{asu_1}, it holds
\[u\geq A^{-\frac{\gamma-1}{2}}\eta,\]
where we also use the definition of $\eta$ in \eqref{eta_def}.
We thus conclude that
\begin{equation}\label{Kbound}
k^+_2
\leq\frac{\gamma-1}{8}A^{-\frac{\gamma+1}{2}}\dot{A}(-2+3(\gamma-1))\eta=\frac{(\gamma-1)(3\gamma-5)}{8}A^{-\frac{\gamma+1}{2}}\dot{A}\eta,
\end{equation}
and
\begin{equation}\label{Kbound2}
k^-_2
\leq\frac{\gamma-1}{8}A^{-\frac{\gamma+1}{2}}\dot{A}(-2-3(\gamma-1))\eta=-\frac{(\gamma-1)(3\gamma-1)}{8}A^{-\frac{\gamma+1}{2}}\dot{A}\eta<0.
\end{equation}
Furthermore, we find
\begin{equation}\label{Kbound3}
k^+_2\leq0\quad\hbox{when}\quad  1<\gamma\leq\frac{5}{3}.
\end{equation}

Using \eqref{F_in}, \eqref{Kbound2} and \eqref{Kbound3}, when $x'>0$, we have
\beq\label{Ri2}
\left\{
\begin{split}
\partial_+\alpha< k_1(\alpha\beta-\alpha^2)+k_2^+(\alpha-\beta),\\[2mm]
\partial_-\beta< k_1(\alpha\beta-\beta^2)+k_2^-(\beta-\alpha).
\end{split}
\right.
\eeq

Consider solutions on the region $(x',t')\in (0,+\infty)\times[0,T)$.
Note when $1<\gamma\leq\frac{5}{3}$, the functions $k^\pm_2$ are both negative.
Then we can directly use the idea in \cite{G9} to see that $\max(\alpha,\beta)<M$ is an invariant domain on the $(\alpha,\beta)$-plane, as shown in Figure \ref{Fig:keyproof}. In fact, it is easy to show that on the right boundary of $\{\max(\alpha,\beta)<M\}$ (except the vertex) where $\alpha=M$ and $\beta<M$, $\alpha$ strictly decays on time, i.e. $\partial_+\alpha<0$. Similarly, $\alpha$ strictly decays on time, i.e. $\partial_-\beta<0$ on the upper boundary of $\{\max(\alpha,\beta)<M\}$. As a direct consequence, $\alpha(x', t')$ and $\beta(x', t')$ have a constant upper bound $\displaystyle\max_{x'\in\mathbb{R^+}}(\alpha(x',0),\beta(x',0))$, so by \eqref{La},
\beq\label{v_key}
v_{t'}=u_{x'}=\frac{1}{2}(\alpha+\beta)<M,
\eeq
which gives a uniform linear time-dependent upper bound on $v$.

\begin{figure}[h!]
\centering
 \scalebox{0.4}{\includegraphics{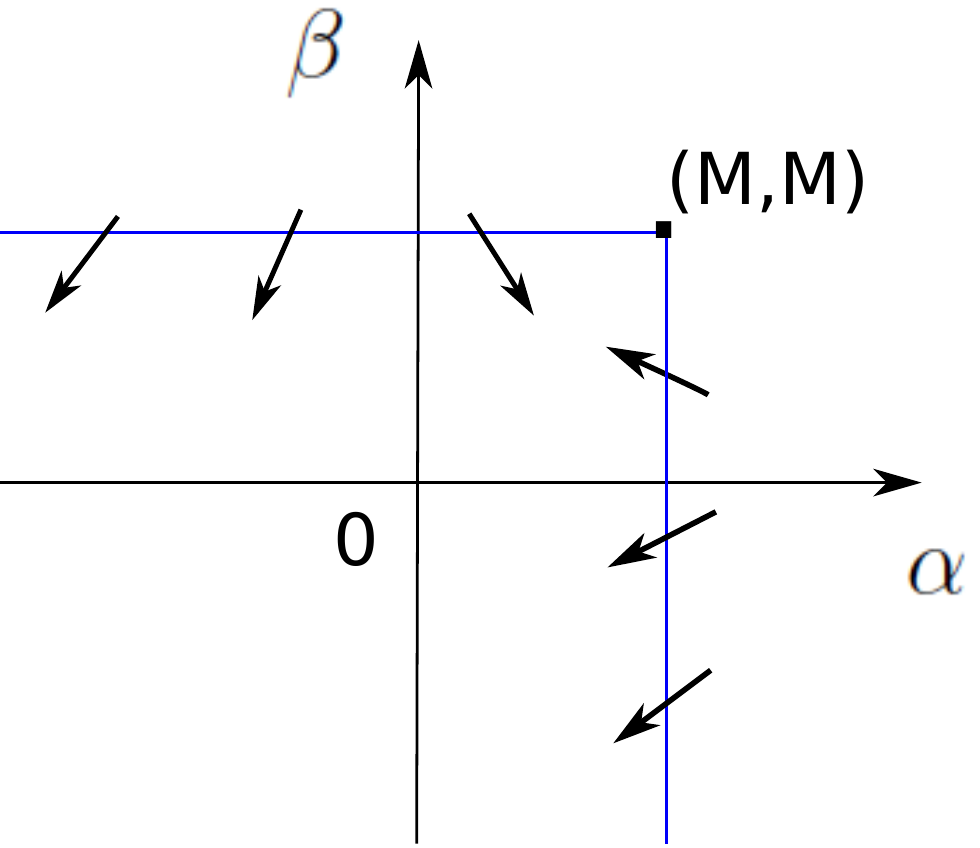}}
\caption{For any positive constant $M$,  the region $\{\max(\alpha,\beta)<M\}$ is an invariant domain when $1<\gamma\leq\frac{5}{3}$.}\label{Fig:keyproof}
\normalsize
\end{figure}


Unfortunately, for the case when  $\frac{5}{3}<\gamma<3$, the coefficient $k^+_2$ might be positive when $d=2,3$, hence
  $\{\max(\alpha,\beta)<M\}$ is not an invariant domain. Instead, we will introduce some new transformation to conquer the possible growth given by $k^+_2$ when $r$ is uniformly away from the origin. Then we can  obtain some non-uniform upper bound on $\alpha$ and $\beta$. Finally, one can also find some non-uniform upper bound on $v$ using \eqref{v_key}.

We will prove two theorems for the upper bound of $v$ for classical solutions to \eqref{La} in the next two subsections: Theorem \ref{thm_den1} when
$1<\gamma\leq\frac{5}{3}$ and Theorem \ref{thm_den2} when $\frac{5}{3}<\gamma<3$.
The upper bounds on $v$ could directly give a lower bound on density depending on $b>0$, when $r\in (b,\infty)$.


\subsection{Upper bound on $v$ when $1<\gamma\leq\frac{5}{3}$}\label{sub_4.1}
By \eqref{Ri2}, we can prove the following lemma. The proof is similar as the one in \cite{G9}. Since the proof is brief, we add the proof to make the paper self-contained.


 \begin{lma}\label{lem_den1}
 Suppose the initial conditions in Assumption \ref{asu_1} are satisfied. Let $1<\gamma\leq\frac{5}{3}$ and the positive constant $M$ be an upper bound of $\alpha(x',0)$ and $\beta(x',0)$, i.e.
 \begin{equation*}\label{M1}
 \max_{x'\in \mathbb{R^+}}\Big\{\alpha(x',0),\beta(x',0)\Big\}<M,
 \end{equation*}
 then for any $t'>0$,
  \begin{equation*}\label{M2}
 \max_{x'\in \mathbb{R^+}}\{\alpha(x',t'),\beta(x',t')\}<M.
 \end{equation*}
 \end{lma}
 \begin{proof}
 We now prove Lemma \ref{lem_den1} by contradiction. Without loss of generality, we may assume $\alpha(x'_*,t'_*)=M$, at some point $(x'_*,t'_*)$. Because wave speed $C$ is bounded on $[0,t'_*]$, then we can find the characteristic triangle with vertex $(x'_*,t'_*)$ and lower boundary on the initial line $t'=0$, denoted by $\Omega_0$. Also by Theorem \ref{thm_bounds}, we know that $\Omega_0$ will not include any part of the line $x'=0$ (equivalently $r=0$).
 See Figure \ref{Fig:1}.
\begin{figure}[h!]
\centering
\scalebox{0.5}{\includegraphics{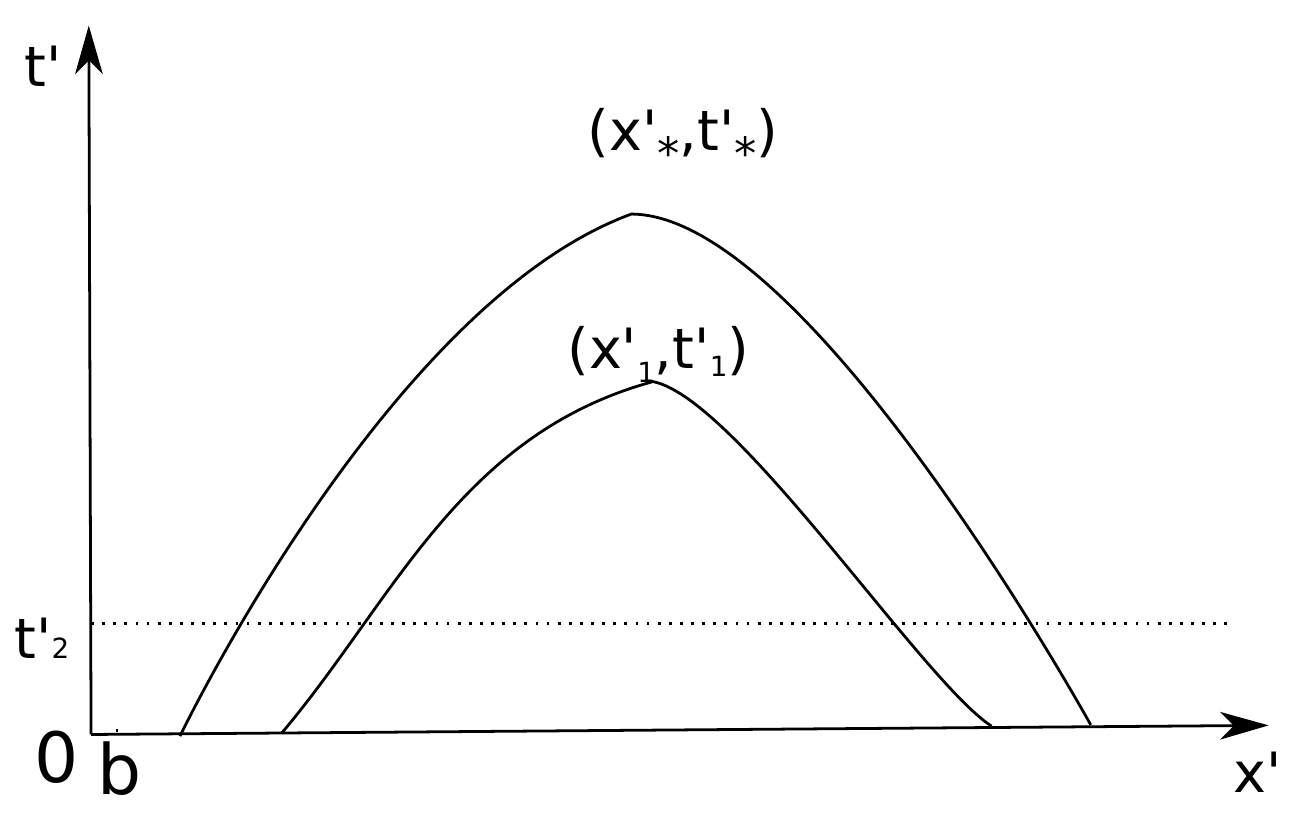}}
 \caption{Proof of Lemma \ref{lem_den1}.  }\label{Fig:1}
\normalsize
\end{figure}
In turn, we can find the first time $t'_1$ such that $\alpha(t'_1)=M$ or $\beta(t'_1)=M$ in $\Omega_0$. More precisely,
 \[\max_{(x',t')\in \Omega_0, t'\in[0,t'_1)}\{\alpha(x',t'),\beta(x',t')\}<M,\]
 with $\alpha(x'_1, t'_1)=M$ or/and $\beta(x'_1, t'_1)=M$ for some $(x'_1, t'_1)\in\Omega_0$. Without loss of generality, we still assume \[\alpha(x'_1, t'_1)=M.\] The other case can be proved similarly. Denote the characteristic triangle with vertex $(x'_1, t'_1)$ by $\Omega_1\subset \Omega_0$, then
 \[\max_{(x',t')\in \Omega_1, t'\in[0,t'_1)}\{\alpha(x',t'),\beta(x',t')\}<M,\]
and $\alpha(x'_1, t'_1)=M$. By the continuity of $\alpha$ and $\beta$, we could find a time $t'_2\in[0,t'_1)$ such that, for any $(x',t')\in\Omega_1, t'_2\leq t'<t'_1$, it holds
\begin{equation*}\label{ALB}
0<\alpha(x',t')< M .
\end{equation*}
Hence, by $\eqref{Ri2}_1$,  along the forward characteristic segment through $(x'_1,t'_1)$, when $t'_2\leq t'<t'_1$, we have, since $k_1>0$ and $k_2^+\leq 0$,
\begin{equation*}\label{A1E}
\partial_+\alpha< (k_1\alpha-k_2^+)(\beta-\alpha)\leq  (k_1\alpha-k_2^+)(M-\alpha)\leq
 K_1(M-\alpha),
\end{equation*}
for some positive constant $K_1$ depending on $M$ and the minimum value of $r$ on the piece of forward characteristic  when $t'_2\leq t'<t'_1$. Note to find $K_1$, we use that fact that both forward and backward characteristics have positive wave speeds in the Eulerian coordinates, so the lowest $r$ value on the characteristic segment considered is at $t=t'_2$. This gives, through the integration along the forward characteristic,
\beq\label{proof_1}
\ln \frac{1}{M-\alpha(t')}\leq\ln \frac{1}{M-\alpha(t'_2)}+\textstyle K_1 (t'-t'_2)\,.
\eeq
As $t'\rightarrow t'_1-$, the left hand side approaches infinity
while the right hand side approaches a finite number, which gives a contradiction.
This completes the proof of Lemma \ref{lem_den1}.
\end{proof}

Using the argument in \eqref{v_key} and Lemma \ref{lem_den1}, we immediately prove the upper bound for $v$ when $1<\gamma\leq\frac{5}{3}$ in the following theorem.

\begin{thm}\label{thm_den1}
We consider the $C^1$ solutions
$(v,u)(x',t')$ of \eqref{La} in the region $(x',t')\in (0,+\infty)\times[0,T)$, for some $T>0$, with initial data satisfying conditions in Assumption \ref{asu_1} and $\alpha(x',0)$, $\beta(x',0)$ are uniformly bounded above by a constant $M$, where $\alpha$ and $\beta$  take the form in \eqref{ab1}. If $1<\gamma\leq\frac{5}{3}$, then we have
\begin{equation*}\label{v1}
v(x',t')\leq v(x',0)+Mt'.
\end{equation*}
 \end{thm}

\subsection{Upper bound on $v$ when $\frac{5}{3}<\gamma<3$}\label{sub_4.2}
Similar as Theorem \eqref{thm_den1}, when $\frac{5}{3}<\gamma<3$,  the key point is still to get the uniform upper bound of some gradient variables measuring rarefaction. However, we fail to control the lower order term $k_2^+(\alpha-\beta)$ by directly using the Riccati equations $\eqref{Ri}$.

The new method is to introduce some transformation on gradient variables $\alpha$ and $\beta$:

\begin{equation}\label{ab2}\tilde{\alpha}:=\eta^\frac{3-\gamma}{\gamma-1}\alpha \quad{\rm and}\quad\tilde{\beta}:=\eta^\frac{3-\gamma}{\gamma-1}\beta .\end{equation}
According to \eqref{3.2} and \eqref{Ri}, we obtain
\begin{equation*}
\begin{cases}
\partial_+\tilde{\alpha}=\frac{3\gamma-5}{2(\gamma-1)}K_c\eta\tilde{\alpha}\tilde{\beta}-\frac{\gamma+1}{2(\gamma-1)}K_c\eta\tilde{\alpha}^2+k_2^+(\tilde{\alpha}-\tilde{\beta})+\eta^\frac{3-\gamma}{\gamma-1} F(x',t'),\\[2mm]
\partial_-\tilde{\beta}=\frac{3\gamma-5}{2(\gamma-1)}K_c\eta\tilde{\alpha}\tilde{\beta}-\frac{\gamma+1}{2(\gamma-1)}K_c\eta\tilde{\beta}^2+k_2^-(\tilde{\beta}-\tilde{\alpha})+\eta^\frac{3-\gamma}{\gamma-1} F(x',t').
\end{cases}
\end{equation*}
That is
\begin{equation}\label{NRi}
\begin{cases}
\partial_+\tilde{\alpha}=
\eta\tilde{K}_1\tilde{\alpha}(\tilde{\beta}-\tilde{\alpha})
-\eta\big(\tilde{K}_2\tilde{\alpha}^2-\tilde{k}^+(\tilde{\alpha}-\tilde{\beta})\big)+\eta^\frac{3-\gamma}{\gamma-1}F(x',t'),\\[2mm]
\partial_-\tilde{\beta}=\eta\tilde{K}_1\tilde{\beta}(\tilde{\alpha}-\tilde{\beta})
-\eta\big(\tilde{K}_2\tilde{\beta}^2-\tilde{k}^-(\tilde{\beta}-\tilde{\alpha})\big)+\eta^\frac{3-\gamma}{\gamma-1} F(x',t'),
\end{cases}
\end{equation}
with
\[\tilde{K}_1=\frac{3\gamma-5}{2(\gamma-1)}K_c,\quad
\tilde{K}_2=\frac{3-\gamma}{\gamma-1}K_c,\]
and
\[\tilde{k}^\pm=-\frac{(\gamma-1)u}{4\eta}A^{-1}\dot{A}\pm\frac{3(\gamma-1)^2}{8} A^{-\frac{\gamma+1}{2}}\dot{A}.\]

Similar as in \eqref{Kbound} and \eqref{Kbound2}, we know that when $x'>0$ or $r(x',t')>0$
\[\tilde{k}^-(x',t')<0,\]
but
$\tilde{k}^+(x',t')$ might be positive.
The nice thing happens in $\eqref{NRi}_1$ is that in any domain of dependence away from the origin, when $\tilde\alpha$ is large enough (depending on the domain) then the second term in $\eqref{NRi}_1$ becomes negative. Hence, the value of $\tilde\alpha$ will not become very large.

Now we define the domain of dependence $\Omega$ as in Figure \ref{Fig:domain}:
\beq\label{ome}
\Omega=\{(x',t')\,|\,x'\geq x^+(t';\bar x'),\ t'\in \mathbb R^+\},
\eeq
where the left boundary characteristic of $\Omega$ is a 2-characteristic starting from the point $(\bar x',0)$, i.e. the characteristic satisfies
\[
\frac{d}{dt'} x^+(t';\bar x')=C(x^+(t';\bar x'),t'),\qquad  x^+(0;\bar x')=\bar x'.
\]

 \begin{figure}[h!]
\centering
 \scalebox{0.4}{\includegraphics{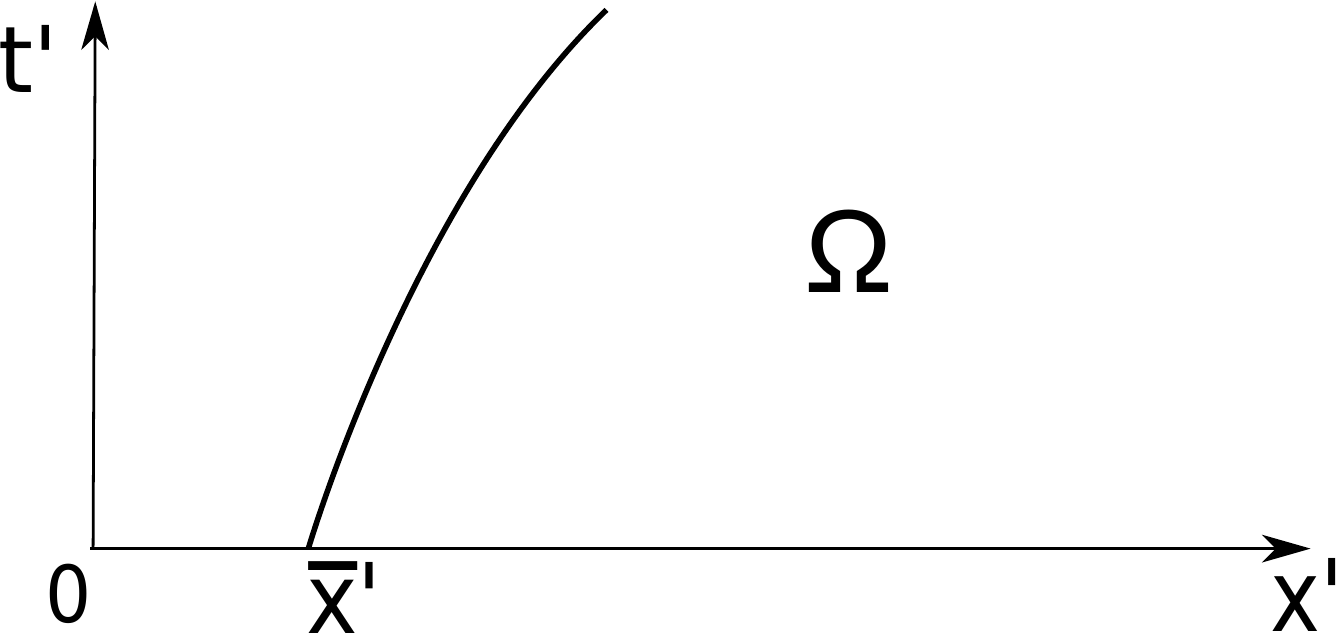}}
\caption{Upper bound on $v$ when $\frac{5}{3}<\gamma<3$ in the domain of dependence $\Omega$. }\label{Fig:domain}
\normalsize
\end{figure}

We first recall that along any 1 or 2-characteristic, the radius $r$ is not decreasing on time since $c_1$ and $c_2$ in the Eulerian coordinates are always nonnegative by \eqref{inbo3} when $1<\gamma\leq 3$, and the characteristics in Eulerian and Lagrangian coordinates are equivalent under transformation. So
\beq\label{b_def}
\min_{(x',t')\in\Omega}r(x',t')=r(\bar x', 0),
\eeq
where we also use the fact that $r(x', 0)>r(\bar x',0)$ when $x'>\bar x'$.
By Theorem \ref{thm_bounds}, it holds that
\begin{equation}\label{Kbound4}
\begin{split}
\tilde{k}^+(x',t')\leq\frac{(\gamma-1)(3\gamma-5)}{8}A^{-\frac{\gamma+1}{2}}\dot{A}&\leq \frac{(\gamma-1)(3\gamma-5)(d-1)}{8}\,r(\bar x', 0)^{\frac{d(1-\gamma)+\gamma-3}{2}}\\
& =:\hat{K}(\bar x'),
\end{split}\end{equation}
when $(x',t')\in \Omega$ and $\frac{5}{3}<\gamma<3$. We will sometimes use $\hat{K}$ to denote $\hat{K}(\bar x')$ if there is no ambiguity.

\begin{definition}\label{def3}
When $\frac{5}{3}<\gamma<3$, we use $N(\bar x')$ to denote a constant depending on $\bar x'$ larger than both
 \begin{equation*}
 \max_{x'\in \Omega}\Big\{\tilde{\alpha}(x',0),\tilde{\beta}(x',0)\Big\}
 \end{equation*}
and
 \begin{equation}\label{Mbound}
\textstyle\max\Big\{\frac{8(\gamma-1)}{(3-\gamma)K_c},\frac{4(\gamma-1)}{(3\gamma-5)K_c}\Big\}\cdot\hat{K}(\bar x'),
 \end{equation}
 where $\hat{K}(\bar x')$  is defined in \eqref{Kbound4}.
\end{definition}

 \begin{thm}\label{thm_den2} Suppose the initial conditions in Assumption \ref{asu_1} are satisfied. Let  $\frac{5}{3}<\gamma<3$.
 We consider the $C^1$ solution in the domain $\Omega$ when $0<t<T$ defined in \eqref{ome}, with left-below vertex $(\bar x',0)$. Then
  \begin{equation}\label{4.27}
 \max_{(x',t')\in\Omega}\{\tilde{\alpha}(x',t'),\tilde{\beta}(x',t')\}<N(\bar x'),
 \end{equation}
 for any  $0<t'<T$. Here $N(\bar x')$ is independent of $T$, and $T$ can be infinity.

 And for any $(x',t')\in \Omega$ with $0<t'<T$,
 \begin{equation}\label{v2}
v^{\frac{\gamma-1}{2}}(x',t')\leq v^{\frac{\gamma-1}{2}}(x',0)+\sqrt{K\gamma}K_\tau^{-1}N(\bar x')t'.
\end{equation}
 \end{thm}
 \begin{proof}
For brief, we just use $N$ to denote $N(\bar x')$ in the proof of this lemma. To guide the proof, the readers can still use Figure \ref{Fig:1}.

 We first prove \eqref{4.27} by contradiction. Without loss of generality, assume that $\tilde{\alpha}(x'_*,t'_*)=N$, at some point $(x'_*,t'_*)$. Since wave speed $C$ is bounded on $[0,t'_*]$, we can find the characteristic triangle with vertex $(x'_*,t'_*)$ and lower boundary on the initial line $t'=0$, denoted by $\Omega_2\subset \Omega$.
 Then we can find the first time $t'_1$ such that $\tilde{\alpha}(t'_1)=N$ or $\tilde{\beta}(t'_1)=N$ in $\Omega_2$. More precisely,
 \[\max_{(x',t')\in \Omega_2, t'\in[0,t'_1)}\{\tilde{\alpha}(x',t'),\tilde{\beta}(x',t')\}<N,\]
 with $\tilde{\alpha}(x'_1, t'_1)=N$ or/and $\tilde{\beta}(x'_1, t'_1)=N$ for some $(x'_1, t'_1)\in\Omega_2$. Without loss of generality, we still assume
 \[\tilde{\alpha}(x'_1, t'_1)=N.\]
 Denote the characteristic triangle with vertex $(x'_1, t'_1)$ by $\Omega_3\subset \Omega_2$, then
 \[\max_{(x',t')\in \Omega_3, t'\in[0,t'_1)}\{\tilde{\alpha}(x',t'),\tilde{\beta}(x',t')\}<N,\]
and $\tilde{\alpha}(x'_1, t'_1)=N$. Now, we divide the problem into two cases:

{\it Case I. $-\frac{N}{2}<\tilde{\beta}(x',t')<N$.} By the continuity of $\tilde{\alpha}$ and $\tilde{\beta}$ and our construction, we could find a time $t'_2\in[0,t'_1)$ such that, for any $(x',t')\in\Omega_3, t'_2\leq t'<t'_1$, we have
\begin{equation}\label{casei}
\frac{N}{2}<\tilde{\alpha}(x',t')< N, \quad |\tilde{\beta}(x',t')|<N.
\end{equation}
Using \eqref{F_in}, $\eqref{NRi}_1$,  \eqref{Kbound4}, \eqref{Mbound}  and  \eqref{casei}, along the forward characteristic segment through $(x'_1,t'_1)$, when $t'_2\leq t'<t'_1$, one has
\begin{equation*}\label{caseie}
\begin{split}
\partial_+\tilde{\alpha}&\textstyle\leq \tilde{K}_1\eta(\tilde{\alpha}\tilde{\beta}-\tilde{\alpha}^2)+\eta\Big(-\frac{3-\gamma}{\gamma-1}K_c\tilde{\alpha}^2+\tilde{k}^+(\tilde{\alpha}-\tilde{\beta})\Big)\\[2mm]
&\leq \tilde{K}_1\eta(\tilde{\alpha}\tilde{\beta}-\tilde{\alpha}^2)\leq \tilde{K}_1\eta\tilde{\alpha}(N-\tilde{\alpha})\leq  L N(N-\tilde{\alpha}).
\end{split}
\end{equation*}
for some positive constant $L$.
Then similar as \eqref{proof_1}, we can find a contradiction.

{\it Case II. $\tilde{\beta}(x',t')\leq -\frac{N}{2}.$} By the continuity of $\tilde{\alpha}$ and $\tilde{\beta}$ and our construction, we could find a time $t'_3\in[t'_0,t'_1)$ such that
\begin{equation}\label{caseii}
\frac{N}{2}<\tilde{\alpha}(x',t')< N\quad{\rm and} \quad \tilde{\beta}<-\frac{N}{4}, \quad {\rm for~ any~ } (x',t')\in\Omega_3,\  t'_3\leq t'<t'_1.
\end{equation}
Thus, using \eqref{Kbound4}, \eqref{Mbound}
and \eqref{caseii}, we obtain
\[\tilde{K}_1\tilde{\alpha}\tilde{\beta}-\tilde{k}^+\tilde{\beta}=(\tilde{K}_1\tilde{\alpha}-\tilde{k}^+)\tilde{\beta}<0.\]
Using this inequality and also by \eqref{F_in}, $\eqref{NRi}_1$, \eqref{Kbound4}, \eqref{Mbound} and  \eqref{caseii}, we have
\begin{equation*}\label{caseiie}
\textstyle\partial_+\tilde{\alpha}\leq\eta\Big(-\frac{\gamma+1}{2(\gamma-1)}K_c\tilde{\alpha}^2+\tilde{k}^+\tilde{\alpha}\Big)<0,
\end{equation*}
which contradicts to that $\tilde{\alpha}(x'_1,t'_1)=N$ while $\tilde{\alpha}(x',t')< N$ when $(x',t')\in\Omega_3,\ t'_3\leq t'<t'_1$. This completes the proof of \eqref{4.27}.

Finally, we prove the upper bound on $v$.
By \eqref{La} and \eqref{4.27}, we have
\begin{equation*}
\eta^\frac{3-\gamma}{\gamma-1}v_{t'}=\eta^\frac{3-\gamma}{\gamma-1}u_{x'}=\frac{1}{2}(\tilde{\alpha}+\tilde{\beta})<N(\bar x'),
\end{equation*}
in $\Omega$,
that is
\[
\textstyle(v^{\frac{\gamma-1}{2}})_{t'}\leq \frac{\gamma-1}{2}\big(\frac{\gamma-1}{2\sqrt{K\gamma}}\big)^\frac{3-\gamma}{\gamma-1}N(\bar x')\]
which directly gives \eqref{v2}.
\end{proof}


\section{Singularity formation}\label{sec_5}
Now we are in a position to study the singularity formation for system \eqref{La}. Here we skip the case when $d=1$, where a complete resolution on the shock formation can be found in \cite{G9,CPZ}. So in this section, $d=2$ or $3$.

In what follows, we will establish the finite time singularity formation results for $1<\gamma\leq 3$ in multiple  space dimension when the initial compression is slightly stronger than a critical value. In order to control the solutions of the decoupled Riccati equations, we shall use the estimates in Theorem \ref{thm_den1} and \ref{thm_den2}. We will follow the notations used in section \ref{sec_4}.

\subsection{Singularity formation for $\gamma\neq \frac{5}{3}$ and $\gamma\neq 3$}\label{subsub_1}
As in \cite{G3, G8}, system \eqref{Ri} can be rewritten as a decoupled system with varying coefficients. We define when $\gamma\neq \frac{5}{3}$ and $\gamma\neq 3$,
\begin{equation}\label{YQ}
\begin{split}
&\textstyle Y:=\eta^{\frac{\gamma+1}{2(\gamma-1)}}\alpha+\frac{(\gamma-1)^2}{2K_c(\gamma-3)}A^{-1}\dot{A}u\eta^{\frac{\gamma-3}{2(\gamma-1)}}-\frac{(3\gamma-13)(\gamma-1)^3}{4K_c(\gamma-3)(3\gamma-5)}A^{-\frac{\gamma+1}{2}}\dot{A}\eta^{\frac{3\gamma-5}{2(\gamma-1)}},\\[2mm]
&\textstyle Q:=\eta^{\frac{\gamma+1}{2(\gamma-1)}}\beta+\frac{(\gamma-1)^2}{2K_c(\gamma-3)}A^{-1}\dot{A}u\eta^{\frac{\gamma-3}{2(\gamma-1)}}+\frac{(3\gamma-13)(\gamma-1)^3}{4K_c(\gamma-3)(3\gamma-5)}A^{-\frac{\gamma+1}{2}}\dot{A}\eta^{\frac{3\gamma-5}{2(\gamma-1)}},
\end{split}
\end{equation}
which satisfy equations
\begin{equation}\label{DRi}
\begin{cases}
\partial_+Y=d_0+d_1Y+d_2Y^2,\\
\partial_-Q=\bar{d}_0+\bar{d}_1Q+d_2Q^2,
\end{cases}
\end{equation}
with coefficients
\begin{equation*}
\begin{split}
d_2&=-\frac{\gamma+1}{2(\gamma-1)}K_c \eta^{\frac{3-\gamma}{2(\gamma-1)}},\\
d_1&=L_1A^{-1}\dot{A}u+L_2A^{-\frac{\gamma+1}{2}}\dot{A}\eta,\\
d_0&=L_3A^{-\frac{\gamma+3}{2}}\dot{A}^2u\eta^{\frac{3\gamma-5}{2(\gamma-1)}}+L_4A^{-\frac{\gamma+1}{2}}\ddot{A}u\eta^{\frac{3\gamma-5}{2(\gamma-1)}}+L_5A^{-(\gamma+1)}\dot{A}^2\eta^{\frac{5\gamma-7}{2(\gamma-1)}}\\
&\quad+L_6A^{-\gamma}\ddot{A}\eta^{\frac{5\gamma-7}{2(\gamma-1)}}
+L_7A^{-1}\ddot{A}u^2\eta^{\frac{\gamma-3}{2(\gamma-1)}}
+L_8A^{-2}\dot{A}^2u^2\eta^{\frac{\gamma-3}{2(\gamma-1)}},\\
\bar{d}_1&=\bar{L}_1A^{-1}\dot{A}u+\bar{L}_2A^{-\frac{\gamma+1}{2}}\dot{A}\eta,\\
\bar{d}_0&=\bar{L}_3A^{-\frac{\gamma+3}{2}}\dot{A}^2u\eta^{\frac{3\gamma-5}{2(\gamma-1)}}+\bar{L}_4A^{-\frac{\gamma+1}{2}}\ddot{A}u\eta^{\frac{3\gamma-5}{2(\gamma-1)}}+\bar{L}_5A^{-(\gamma+1)}\dot{A}^2\eta^{\frac{5\gamma-7}{2(\gamma-1)}}\\
&\quad+\bar{L}_6A^{-\gamma}\ddot{A}\eta^{\frac{5\gamma-7}{2(\gamma-1)}}+\bar{L}_7A^{-1}\ddot{A}u^2\eta^{\frac{\gamma-3}{2(\gamma-1)}}+\bar{L}_8A^{-2}\dot{A}^2u^2\eta^{\frac{\gamma-3}{2(\gamma-1)}},\\
\end{split}
\end{equation*}
where $L_j$ and $\bar{L}_j$, $i,j=1,2\cdots,8$ are constants depending only on $\gamma$. The derivation of this system can be found in \cite{G8}. Then system \eqref{DRi} can be further  written as
\begin{equation}\label{DRIE}
\begin{cases}
\partial_+Y=d_0+d_1Y+\frac{1}{2}d_2Y^2+\frac{1}{2}d_2Y^2,\\[2mm]
\partial_-Q=\bar{d}_0+\bar{d}_1Q+\frac{1}{2}d_2Q^2+\frac{1}{2}d_2Q^2.
\end{cases}
\end{equation}
 The roots of $d_0+d_1Y+\frac{1}{2}d_2Y^2=0$ and $\bar{d}_0+\bar{d}_1Q+\frac{1}{2}d_2Q^2=0$, if they exist, are given by the quadratic formula
\begin{equation}\label{Y}
Y_\pm=\frac{-d_1\pm\sqrt{d_1^2-2d_0d_2}}{d_2}
\end{equation}
 and \begin{equation}\label{Q}
 Q_\pm=\frac{-\bar{d}_1\pm\sqrt{\bar{d}_1^2-2\bar{d}_0 d_2}}{d_2}.\end{equation}
 The structure of  the roots $Y_\pm$ and $Q_\pm$ lead us to study the ratios $\frac{d_0}{d_1}$ and $\frac{d_0}{d_2}$ which dominate behaviors of solutions. A direct calculation is carried out as follows.
Since $A=r^{d-1}$ and $\dot{A}=\frac{dA(r)}{dr}$, the ratios are
\begin{equation*}
\begin{split}
\frac{d_1}{d_2}
&
=r^{-1}\eta^{\frac{\gamma-3}{2(\gamma-1)}}(L_9u+L_{10} r^{\frac{(1-\gamma)(d-1)}{2}}\eta),\\[2mm]
\frac{d_0}{d_2}
&=r^{-2}\eta^{\frac{\gamma-3}{\gamma-1}}(L_{11}r^{\frac{(1-\gamma)(d-1)}{2}}u\eta+L_{12}r^{(1-\gamma)(d-1)}\eta^2+\hat{L}_5u^2),
\end{split}\end{equation*}
when $d>1$, where $L_i$ are constants depending only on $\gamma$ and $d$. Thus, the roots
\begin{equation}\label{YE}
\begin{split}
Y_\pm&=\frac{-d_1\pm\sqrt{d_1^2-2d_0d_2}}{d_2}\geq -r^{-1}\eta^{\frac{\gamma-3}{2(\gamma-1)}}(L_{13}u+L_{14}r^{\frac{(1-\gamma)(d-1)}{2}}\eta),
\end{split}\end{equation}
for some positive constants $L_{13}$ and $L_{14}$ depending only on $\gamma$ and $d$. Similar bounds hold for $Q_{\pm}$.

In a summary, we can find two positive constants $L_{15},L_{16}$ depending only on $\gamma$ and $d$, such that
\begin{equation}\label{YQE}
\begin{split}
-\frac{2d_1}{d_2},-\frac{2\bar{d}_1}{d_2},Y_\pm,Q_\pm&\geq -r^{-1}\eta^{\frac{\gamma-3}{2(\gamma-1)}}(L_{15}u+L_{16} r^{\frac{(1-\gamma)(d-1)}{2}}\eta).
\end{split}\end{equation}

Now we prove the singularity formation theorem. First, we define
\begin{equation}\label{GT}
G_T(\bar{x}')\equiv G_{\gamma,T}(\bar{x}'):=
\begin{cases}
C_1\,(r(\bar{x}',0))^{-1}K_\tau^{\frac{\gamma-3}{4}}\big(\bar{v}+MT\big)^{\frac{3-\gamma}{4}}
, \quad{\rm if}\quad 1<\gamma<\frac{5}{3},\\[2mm]
C_1\,(r(\bar{x}',0))^{-1}K_\tau^{\frac{\gamma-3}{4}}\big(\bar{v}^\frac{\gamma-1}{2}
+\sqrt{K\gamma}K^{-1}_\tau N(\bar{x}')T\big)^{\frac{3-\gamma}{2(\gamma-1)}},\\
\quad\quad\quad\quad\quad\quad\quad\quad\quad\quad\quad\quad\quad\quad\quad\quad{\rm if}\quad \frac{5}{3}<\gamma<3,\\
\end{cases}
\end{equation}
with $C_1=2C_0(L_{15}+L_{16})$ and
\[\displaystyle\bar{v}=\frac{1}{\displaystyle\min_{x'\in(\bar{x}',+\infty)}A(r(x',0))\rho(x',0)}.\]
Here it is easy to check that $G_T(\bar x')$ is a decreasing function on $\bar x'$.

The main result reads as follows:

\begin{thm}\label{thm_main}
For $1<\gamma< 3$ and $\gamma\neq \frac{5}{3},$ assume the initial data satisfy conditions in Assumption \ref{asu_1} and
$(\alpha(x',0)$, $\beta(x',0)$, $\tilde{\alpha}(x',0)$, $\tilde{\beta}(x',0))$ are all uniformly bounded in $d=2,3$. For any $C^1$ solution of \eqref{La}, if there exists some $\bar{x}'>0$ and $T$ satisfying
 \begin{equation}\label{Test}
T\geq \frac{4K_{\tau}r(\bar{x}',0)}{C_0(L_{15}+L_{16})(\gamma+1)},
\end{equation}
 such that
\begin{equation}\label{GTcritical}
Y(\bar{x}',0)<-G_T(\bar{x}'),
\end{equation}
then classical solution will break down before $t=T$.
\end{thm}

\begin{proof}
We want to prove that when \eqref{GTcritical} is satisfied, then blowup happens before time $t=T$.

We only have to consider the solution when $t\leq T$, where $T$ satisfies \eqref{Test}.

First, along any $2$-characteristic $\Lambda$ starting from the point $(\bar x',0)$, the estimates on $v$ in Theorems \ref{thm_den1} or \ref{thm_den2} are satisfied, where the region to the right of this $\Lambda$ is the region $\Omega$ defined in \eqref{ome}. One can still use Figure \ref{Fig:domain} as a reference.

 More precisely, at any point $(x',t')$ on the $2$-characteristic $\Lambda$ with $t'<T$, when $1<\gamma<\frac{5}{3}$, Theorem
 \ref{thm_den1}  gives that
\begin{equation*}
v(x',t')\leq v(x',0)+Mt',
\end{equation*}
then using $\eta=\frac{2\sqrt{K\gamma}}{\gamma-1}v^{-\frac{\gamma-1}{2}}$ and \eqref{b_def},  we have
\begin{equation}\label{final}
r^{-1}\eta^{\frac{\gamma-3}{2(\gamma-1)}}(x',t')\leq (r(\bar{x}',0))^{-1}K_\tau^{\frac{\gamma-3}{4}}\big(\bar{v}+MT\big)^{\frac{3-\gamma}{4}}.
\end{equation}\label{final_1}
When $\frac{5}{3}<\gamma<3$, Theorem
 \ref{thm_den2}  gives that
\begin{equation*}
v^{\frac{\gamma-1}{2}}(x',t')\leq v^{\frac{\gamma-1}{2}}(x',0)+\sqrt{K\gamma}K_\tau^{-1}N(\bar x')t',
\end{equation*}
and similarly, by \eqref{b_def}, one obtains
\begin{equation}\label{final_2}
r^{-1}\eta^{\frac{\gamma-3}{2(\gamma-1)}}(x',t')\leq (r(\bar{x}',0))^{-1}K_\tau^{\frac{\gamma-3}{4}}\big(\bar{v}^\frac{\gamma-1}{2}
+\sqrt{K\gamma}K^{-1}_\tau N(\bar x')T\big)^{\frac{3-\gamma}{2(\gamma-1)}}.
\end{equation}


By Theorem \ref{thm_bounds},
\[2C_0\geq u\geq A^{-\frac{\gamma-1}{2}}\eta(r,t),\]
where $A=r^{d-1}$.
Then we have
\begin{equation}\label{l1516}
L_{15}u+L_{16} r^{\frac{(1-\gamma)(d-1)}{2}}\eta\leq 2C_0(L_{15}+L_{16}).
\end{equation}

Based on the above analysis \eqref{final}--\eqref{l1516}, when  $1<\gamma< 3$ and $\gamma\neq \frac{5}{3},$ using \eqref{YQE} and \eqref{GT}, we have
\begin{equation}\label{GTE}
-\frac{2d_1}{d_2},-\frac{2\bar{d}_1}{d_2}>-G_{T}(\bar{x}')\quad{\rm and}\quad Y_\pm,Q_\pm>-G_{T}(\bar{x}').
\end{equation}
Suppose \eqref{GTcritical} is satisfied. 
We want to prove the key estimate
\begin{equation}\label{part_main}
\partial_+Y< \frac{1}{2}d_2Y^2,
\end{equation}
on $\Lambda$ before $t=T$.
In fact, notice that along the 2-characteristic $\Lambda$ starting from $(\bar{x}',0)$, by \eqref{GTE}, we have
\beq\label{5.18}
\partial_+ Y<0\quad \hbox{and}\quad
Y(t')<-G_{T}(\bar{x}'),
\eeq
which in turn implies
\beq\label{Yd1d2}
Y(t')+\frac{2d_1}{d_2}<0.
\eeq

On the part of characteristic when the roots \eqref{Y} do not exist, i.e. when $d_1^2-2d_0d_2<0$, we have $d_0<0$, since $d_2<0$. In view of \eqref{DRIE}, we get
\begin{equation*}\label{D00}
\partial_+ Y=d_0+d_1Y+d_2Y^2<\frac{1}{2}d_2Y(Y+\frac{2d_1}{d_2})+\frac{1}{2}d_2Y^2
<\frac{1}{2}d_2Y^2,
\end{equation*}
by \eqref{Yd1d2}.

On the part of characteristic when the roots \eqref{Y} exist,  we can rewrite the equation $\eqref{DRIE}_1$ as
\begin{equation*}\label{MDRi}
\partial_+Y=\frac{1}{2}d_2(Y-Y_+)(Y-Y_-)+\frac{1}{2}d_2Y^2.
\end{equation*}
By \eqref{5.18} and \eqref{YQE}--\eqref{GT},
\[Y(t')-Y_+<0,\quad Y(t')-Y_-< 0,\]
which implies that
\begin{equation*}\label{part}
\partial_+Y=\frac{1}{2}d_2(Y-Y_+)(Y-Y_-)+\frac{1}{2}d_2Y^2< \frac{1}{2}d_2Y^2,
\end{equation*}
since $d_2<0$. Hence, we proved \eqref{part_main}.

 Integrating \eqref{part_main} in time, we get
\begin{equation}\label{inte1}
\frac{1}{Y(t')}\geq \frac{1}{Y(\bar{x}',0)}-\frac{1}{2}\int_0^{t'}d_2(s)\,ds,
\end{equation}
where the integral is along $\Lambda$. Hence the blowup happens at a time $t'_1$ when the right hand side of \eqref{inte1} equals to zero, that is,
\begin{equation*}
-\frac{1}{Y(\bar{x}',0)}=\frac{1}{2}\int_0^{t'_1}(-d_2)(s)\,ds.
\end{equation*}
In fact, when $1< \gamma<\frac{5}{3}$, we read from the definition of $d_2$ in \eqref{DRi} and Theorem \ref{thm_den1} that
\begin{equation}\label{d2est1}
-d_2\geq \frac{\gamma+1}{4}K_\tau^{-\frac{\gamma+1}{4}}[\bar{v}+Mt']^{-\frac{3-\gamma}{4}}.
\end{equation}
 When $\frac{5}{3}< \gamma<3$, by using Theorem \ref{thm_den2}, we can get
\begin{equation}\label{d2est2}
-d_2\geq \frac{\gamma+1}{4}K_\tau^{-\frac{\gamma+1}{4}}[\bar{v}^{\frac{\gamma-1}{2}}+\sqrt{K\gamma}K_\tau^{-1}N(\bar{x}')t']^{-\frac{3-\gamma}{2(\gamma-1)}}.\end{equation}
Thus, it is clear from the estimates \eqref{d2est1} and \eqref{d2est2} that such a finite time $t'_1$ exists. However, we still need to show that $t'_1<T.$ From \eqref{GTcritical}, we only need to show that
\begin{equation}\label{show1}
\frac{1}{G_{T}(\bar{x}')}\leq \frac{1}{2}\int_0^{T}(-d_2)(s)\,ds.
\end{equation}
When $1< \gamma<\frac{5}{3}$, \eqref{d2est1} gives
 \begin{equation*}
 \begin{split}
 \frac{1}{2}\int_0^{T}(-d_2)(s)\,ds&\geq \frac{\gamma+1}{8}K_\tau^{-\frac{\gamma+1}{4}}\int_0^{T}[\bar{v}+Mt']^{-\frac{3-\gamma}{4}}\,dt\\[2mm]
 & \geq\frac{\gamma+1}{8}K_\tau^{-\frac{\gamma+1}{4}}T(\bar{v}+MT)^{\frac{\gamma-3}{4}}.
\end{split}\end{equation*}
Combining \eqref{GT}  and \eqref{Test}, we then obtain \eqref{show1} for $1< \gamma<\frac{5}{3}$.

When $\frac{5}{3}< \gamma<3$, from \eqref{d2est2}, we obtain
 \begin{equation*}
 \begin{split}
 \frac{1}{2}\int_0^{T}(-d_2)(s)\,ds&\geq\frac{\gamma+1}{8}K_\tau^{-\frac{\gamma+1}{4}}\int_0^{T}[\bar{v}^{\frac{\gamma-1}{2}}+\sqrt{K\gamma}K_\tau^{-1}N(\bar{x}')t']^{-\frac{3-\gamma}{2(\gamma-1)}}\,dt\\[2mm]
 & \geq\frac{\gamma+1}{8}K_\tau^{-\frac{\gamma+1}{4}}T[\bar{v}^{\frac{\gamma-1}{2}}+\sqrt{K\gamma}K_\tau^{-1}N(\bar{x}')T]^{\frac{\gamma-3}{2(\gamma-1)}}.\label{5.23}
\end{split}\end{equation*}
This together with \eqref{GT} and \eqref{Test} imply \eqref{show1} immediately.
This completes the proof of Theorem \ref{thm_main}.
\end{proof}


Finally, we state the singularity formation theorem when the backward initial compression is strong enough. One must assume more initial conditions to make sure that blowup happens before the time when the 1-characteristic including strong compression leaves  $\Omega$ defined in \eqref{ome}. Our main result reads as follows:

\begin{thm}\label{thm_main1}
For $1<\gamma< 3$ and $\gamma\neq \frac{5}{3},$ assume the initial data satisfy conditions in Assumption \ref{asu_1} and
$(\alpha(x',0)$, $\beta(x',0)$, $\tilde{\alpha}(x',0)$, $\tilde{\beta}(x',0))$ are all uniformly bounded in $d=2,3$. For any $C^1$ solution of \eqref{La}, if there exist some $\bar{x}'_0>\bar{x}'>0$ and $T$, with $r(\bar{x}'_0,0)-r(\bar{x}',0)\geq 2C_0\gamma T$ and  $T$ satisfying \eqref{Test},
 such that
\begin{equation}\label{GTcritical1}
 Q(\bar{x}'_0,0)< -G_T(\bar{x}'),
\end{equation}
then classical solution will break down before $t=T$.
\end{thm}
\begin{proof}
We want to prove that singularity formation happens before time $t=T$ when $\eqref{GTcritical1}$ is satisfied.

We only need to check that before time $t=T$, the 1-characteristic starting from $\bar x_0$ will not reach the left boundary of $\Omega$ defined in \eqref{ome}. This is clear because $r(\bar{x}'_0,0)-r(\bar{x}',0)\geq 2C_0\gamma T$, and in the Euler coordinates
\begin{equation*}
\begin{split}
\int_0^Tc_2\,dt-\int_0^Tc_1\,dt&\leq \int_0^T(u+\sqrt{K\gamma}\rho^\frac{\gamma-1}{2})\,dt+\int_0^T\sqrt{K\gamma}\rho^\frac{\gamma-1}{2}\,dt\\
&\leq 2C_0T+2(\gamma-1)C_0 T=2C_0\gamma T
\end{split}\end{equation*}
by Theorem \ref{thm_bounds}, here the first integration of the above inequality is along  the characteristic $\frac{dr}{dt}=c_2$ from $0$ to $T$, the second integration is along the characteristic $\frac{dr}{dt}=c_1$ from $0$ to $T$,  See Figure \ref{Fig:5}.
 \begin{figure}[h!]
\centering
 \scalebox{0.7}{\includegraphics{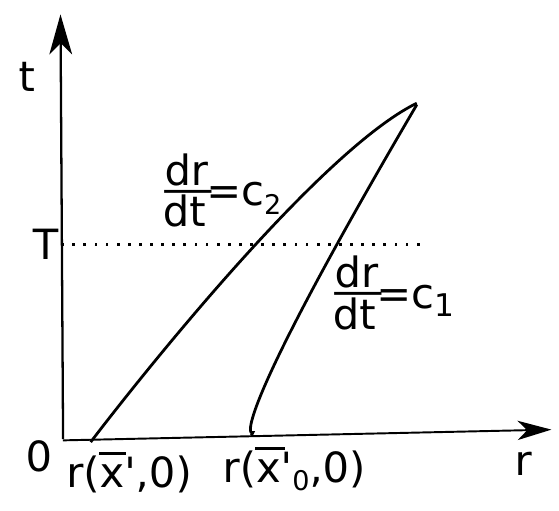}}
\caption{Proof of Theorem \ref{thm_main1}.  }\label{Fig:5}
\normalsize
\end{figure}

Using similar arguments as in Theorem \ref{thm_main}, one can prove this theorem, where we also use the fact that
$ -G_T(\bar{x}')\leq -G_T(\bar{x})$ for any $\bar{x}\geq \bar{x}'$.

\end{proof}

\subsection{Singularity formation for $\gamma=\frac{5}{3}$}\label{subsub_2}
Our present goal is to understand wether the finite time shock formation result also holds when $\gamma=\frac{5}{3}$. As before, we define
\begin{equation}\label{YQ1}
\begin{split}
&Y_1:=\eta^2\alpha-\frac{1}{6K_c}A^{-1}\dot{A}u\eta^{-1}-\frac{1}{3K_c}A^{-\frac{4}{3}}\dot{A}\ln\eta,\\[2mm]
&Q_1:=\eta^2\beta-\frac{1}{6K_c}A^{-1}\dot{A}u\eta^{-1}+\frac{1}{3K_c}A^{-\frac{4}{3}}\dot{A}\ln\eta.
\end{split}
\end{equation}
which satisfy the equations
\begin{equation*}\label{DRi1}
\begin{cases}
\partial_+Y_1=\hat{d}_0+\hat{d}_1Y_1+d_2Y^2_1,\\[2mm]
\partial_-Q_1=\tilde{d}_0+\tilde{d}_1Q_1+d_2Q^2_1,
\end{cases}
\end{equation*}
with coefficients
\begin{equation*}
\begin{split}
d_2&=-2K_c \eta,\\
\hat{d}_1&=\hat{L}_1A^{-1}\dot{A}u+\hat{L}_2A^{-\frac{4}{3}}\dot{A}\eta\ln\eta+\hat{L}_3A^{-\frac{4}{3}}\dot{A}\eta,\\
\hat{d}_0&=\hat{L}_4A^{-2}\dot{A}^2u^2\eta^{-1}+\hat{L}_5A^{-\frac{7}{3}}\dot{A}^2u\ln\eta+\hat{L}_6A^{-\frac{7}{3}}\dot{A}^2u+\hat{L}_7A^{-\frac{8}{3}}\dot{A}^2\eta\ln^2\eta\\
&\quad+\hat{L}_8A^{-\frac{8}{3}}\dot{A}^2\eta\ln\eta+\hat{L}_{9}A^{-\frac{8}{3}}\dot{A}^2\eta+\hat{L}_{10}A^{-1}\ddot{A}u^2\eta^{-1}+\hat{L}_{11}A^{-\frac{4}{3}}\ddot{A}u\\
&\quad+\hat{L}_{12}A^{-\frac{4}{3}}\ddot{A}u\ln\eta+\hat{L}_{13}A^{-\frac{5}{3}}\ddot{A}\eta+\hat{L}_{14}A^{-\frac{5}{3}}\ddot{A}\eta\ln\eta,\\
\tilde{d}_1&=\tilde{L}_1A^{-1}\dot{A}u+\tilde{L}_2A^{-\frac{4}{3}}\dot{A}\eta\ln\eta+\tilde{L}_3A^{-\frac{4}{3}}\dot{A}\eta,\\
\tilde{d}_0&=\tilde{L}_4A^{-2}\dot{A}^2u^2\eta^{-1}+\tilde{L}_5A^{-\frac{7}{3}}\dot{A}^2u\ln\eta+\tilde{L}_6A^{-\frac{7}{3}}\dot{A}^2u+\tilde{L}_7A^{-\frac{8}{3}}\dot{A}^2\eta\ln^2\eta\\
&\quad+\tilde{L}_8A^{-\frac{8}{3}}\dot{A}^2\eta\ln\eta+\tilde{L}_{9}A^{-\frac{8}{3}}\dot{A}^2\eta+\tilde{L}_{10}A^{-1}\ddot{A}u^2\eta^{-1}+\tilde{L}_{11}A^{-\frac{4}{3}}\ddot{A}u\\
&\quad+\tilde{L}_{12}A^{-\frac{4}{3}}\ddot{A}u\ln\eta+\tilde{L}_{13}A^{-\frac{5 }{3}}\ddot{A}\eta+\tilde{L}_{14}A^{-\frac{5}{3}}\ddot{A}\eta\ln\eta,\\
\end{split}
\end{equation*}
where $\hat{L}_j$ and $\tilde{L}_j$ are constants. The roots of  $\hat{d}_0+\hat{d}_1Y_1+\frac{1}{2}d_2Y^2_1=0$ and $\tilde{d}_0+\tilde{d}_1Q_1+\frac{1}{2}d_2Q^2_1$, if they exist, are given by the quadratic formula
\begin{equation*}
Y_{1,\pm}=\frac{-\hat{d}_1\pm\sqrt{\hat{d}_1^2-2\hat{d}_0d_2}}{d_2}
\quad{\rm and }\quad
Q_{1,\pm}=\frac{-\tilde{d}_1\pm\sqrt{\tilde{d}_1^2-2\tilde{d}_0\bar{d}_2}}{d_2}.\end{equation*}
Now, applying the same argument as in \eqref{YQE}, \eqref{final}--\eqref{GTE}, for any classical solutions  in $[0,T)$ satisfying initial conditions in Assumption \ref{asu_1},  in view of Theorem \ref{thm_den1}, we can find some positive constants $\hat{L}_{15},\hat{L}_{16}, \hat{L}_{17}$ depending only on $d$, such that, when $\gamma=\frac{5}{3}$,
\begin{equation*}
-\frac{2\hat{d}_1}{d_2},-\frac{2\tilde{d}_1}{d_2},Y_{1,\pm},Q_{1,\pm}\geq -\eta^{-1}[\hat{L}_{15}r^{-1}u+\hat{L}_{16}r^{\frac{-2-d}{3}}\eta+\hat{L}_{17} r^{\frac{-2-d}{3}}\eta\ln\eta]
\end{equation*}
and we define that
\beq\label{YQE1}
G_{\gamma,T}(\bar{x}')= 2C_0K_\tau^{-\frac{1}{3}}(r(\bar{x}',0))^{-1}(\bar v+MT)^{\frac{1}{3}}[\hat{L}_{15}+\hat{L}_{16}+2C_0\hat{L}_{17}r(\bar{x}',0)^{\frac{d-1}{3}}]\\[2mm]
\end{equation}
with
\[\displaystyle\bar{v}=\frac{1}{\displaystyle\min_{x'\in(\bar{x}',+\infty)}A(r(x',0))\rho(x',0)}.\]
and the upper bound $T$ on blowup time, which satisfies
\begin{equation}\label{Test1}
T\geq \frac{3r(\bar{x}',0)K_\tau}{2C_0(\hat{L}_{15}+\hat{L}_{16}+2C_0\hat{L}_{17}r(\bar{x}',0)^{\frac{d-1}{3}})}.
\end{equation}

For $\gamma=\frac{5}{3}$, we have the following gradient blowup results. The proof is similar to Theorem \ref{thm_main} and \ref{thm_main1}, we omit it here for brevity.

\begin{thm}\label{thm_main2}
For $\gamma=\frac{5}{3}$, assume the initial data satisfy conditions in Assumption \ref{asu_1} and $(\alpha(x',0)$, $\beta(x',0))$ are both uniformly bounded in $d=2,3$. Assume one of the following two conditions holds
for $G_{\gamma,T}(\bar{x}')$ defined in \eqref{YQE1}:

(1) There exist some $\bar{x}'>0$ and $T$ satisfying \eqref{Test1}, such that,
\begin{equation*}
Y_1(\bar{x}',0)<-G_{\gamma,T}(\bar{x}'). \end{equation*}

(2)  There exist some $\bar{x}'_0>\bar{x}'>0$ and $T$, with $r(\bar{x}'_0,0)-r(\bar{x}',0)\geq 2C_0\gamma T$  and $T$ satisfying \eqref{Test1}, such that,
\begin{equation*}
Q_1(\bar{x}'_0,0)<-G_{\gamma,T}(\bar{x}').
\end{equation*}

Then the $C^1$ solution of \eqref{La} will break down before time $T$.
\end{thm}

\subsection{Singularity formation for $\gamma=3$}
At last, we show the finite time blowup of solutions
when $\gamma=3$. Toward this goal, we define
\begin{equation}\label{YQ2}
\begin{split}
&Y_2:=\eta\alpha+\frac{1}{2K_c}A^{-1}\dot{A}u\ln\eta-\frac{2}{K_c}A^{-2}\dot{A}\eta
+\frac{1}{2K_c}A^{-2}\dot{A}\eta\ln\eta,\\[2mm]
&Q_2:=\eta\beta+
\frac{1}{2K_c}A^{-1}\dot{A}u\ln\eta+\frac{2}{K_c}A^{-2}\dot{A}\eta-\frac{1}{2K_c}A^{-2}\dot{A}\eta\ln\eta,
\end{split}
\end{equation}
which satisfy the equations
\begin{equation*}\label{DRi2}
\begin{cases}
\partial_+Y_2=\check{d}_0+\check{d}_1Y_2+d_2Y^2_2,\\[2mm]
\partial_-Q_2=\underline{d}_0+\underline{d}_1Q_2+d_2Q^2_2,
\end{cases}
\end{equation*}
with coefficients
\begin{equation*}
\begin{split}
d_2&=-K_c,\\
\check{d}_1&=\check{L}_1A^{-1}\dot{A}u+\check{L}_2A^{-1}\dot{A}u\ln\eta+\check{L}_3A^{-2}\dot{A}\eta+\check{L}_4A^{-2}\dot{A}\eta\ln\eta,\\
\check{d}_0&=\check{L}_5A^{-2}\dot{A}^2u^2\ln\eta+\check{L}_{6}A^{-2}\dot{A}^2u^2\ln^2\eta+\check{L}_7A^{-3}\dot{A}^2u\eta+\check{L}_8A^{-3}\dot{A}^2u\eta\ln\eta\\
&\quad+\check{L}_{9}A^{-3}\dot{A}^2u\eta\ln^2\eta+\check{L}_{10}A^{-4}\dot{A}^2\eta^2+\check{L}_{11}A^{-4}\dot{A}^2\eta^2\ln\eta+\check{L}_{12}A^{-4}\dot{A}^2\eta^2\ln^2\eta\\
&\quad+\check{L}_{13}A^{-2}\ddot{A}u\eta\ln\eta+
\check{L}_{14}A^{-2}\ddot{A}u\eta+
\check{L}_{15}A^{-3}\ddot{A}\eta^2\ln\eta+
\check{L}_{16}A^{-3}\ddot{A}\eta^2+
\check{L}_{17}A^{-1}\ddot{A}u^2\ln\eta,\\
\underline{d}_1&=\underline{L}_1A^{-1}\dot{A}u+\underline{L}_2A^{-1}\dot{A}u\ln\eta+\underline{L}_3A^{-2}\dot{A}\eta+\underline{L}_4A^{-2}\dot{A}\eta\ln\eta,\\
\underline{d}_0&=\underline{L}_5A^{-2}\dot{A}^2u^2\ln\eta+\underline{L}_{6}A^{-2}\dot{A}^2u^2\ln^2\eta+\underline{L}_7A^{-3}\dot{A}^2u\eta+\underline{L}_8A^{-3}\dot{A}^2u\eta\ln\eta\\
&\quad+\underline{L}_{9}A^{-3}\dot{A}^2u\eta\ln^2\eta+\underline{L}_{10}A^{-4}\dot{A}^2\eta^2+\underline{L}_{11}A^{-4}\dot{A}^2\eta^2\ln\eta+\underline{L}_{12}A^{-4}\dot{A}^2\eta^2\ln^2\eta\\
&\quad+\underline{L}_{13}A^{-2}\ddot{A}u\eta\ln\eta
+\underline{L}_{14}A^{-2}\ddot{A}u\eta
+\underline{L}_{15}A^{-3}\ddot{A}\eta^2\ln\eta
+\underline{L}_{16}A^{-3}\ddot{A}\eta^2
+\underline{L}_{17}A^{-1}\ddot{A}u^2\ln\eta,\\
\end{split}
\end{equation*}
where $\check{L}_j$ and $\underline{L}_j$ are constants. The roots of $\check{d}_0+\check{d}_1Y_2+\frac{1}{2}d_2Y^2_2=0$ and $\underline{d}_0+\underline{d}_1Q_2+\frac{1}{2}d_2Q^2_2=0$, if they exist, are given by the quadratic formula
\begin{equation*}\label{Y2}
Y_{2,\pm}=\frac{-\check{d}_1\pm\sqrt{\check{d}_1^2-2\check{d}_0d_2}}{d_2}
\quad{\rm and }\quad
Q_{2,\pm}=\frac{-\underline{d}_1\pm
\sqrt{\underline{d}_1^2-2\underline{d}_0d_2}}{d_2}.\end{equation*}
In a similar  procedure as in \eqref{YQE}, it is easy to see that there exists a positive constant $\check{L}_{18}$ depending only on $d$, such that
\begin{equation*}
-\frac{2\check{d}_1}{d_2},-\frac{2\underline{d}_1}{d_2},Y_{2,\pm},Q_{2,\pm}\geq -\check{L}_{18}(r^{-1}u+r^{-1}u\ln\eta+ r^{-d}\eta +r^{-d}\eta\ln\eta).
\end{equation*}
By using Theorem \ref{thm_bounds}, when $\gamma=3$, it holds that
\[2C_0\geq u\geq A^{-\frac{\gamma-1}{2}}\eta(r,t)=r^{1-d}\eta,\]
where $A=r^{d-1}$. We finally obtain
\begin{equation}\label{YQE2}
\begin{split}
-\frac{2\check{d}_1}{d_2},-\frac{2\underline{d}_1}{d_2},Y_{2,\pm},Q_{2,\pm} &\geq-4C_0\check{L}_{18}\big[r(\bar{x}',0)^{-1}+2C_0 r(\bar{x}',0)^{\frac{d-3}{2}}\big]\\
&=:-H(\bar{x}'),
\end{split}\end{equation}
along any 1 or 2 characteristic starting from $(\bar{x}',0)$, on which $r$ always increases with respective to time so large than $r(\bar{x}',0)$.
Here we have used the fact that $\ln\eta\leq \eta^{\frac{1}{2}}$.

Now we come to prove the singularity formation result  for $\gamma=3$: Theorem \ref{thm_main3}. Since $H(\bar{x}')$ is independent of time $T$, Theorem \ref{thm_main3} is much stronger than Theorem \ref{thm_2.4}.

\begin{proof}
Suppose that \eqref{Hcritical} holds.  Without loss of generality, we assume that
$Y_2(\bar{x}',0)<-H(\bar{x}').$ The case when $Q_2(\bar{x}',0)<-H(\bar{x}')$ is similar.
Now consider the  $2$-characteristic $\Lambda$ starting at $(\bar{x}',0)$.

Similar as in Theorem \ref{thm_main}, we have
\begin{equation*}
\partial_+Y_2< -\frac{1}{2}K_cY^2_2
\end{equation*}
on $\Lambda$ for any time before blowup, since $H(\bar{x}')$ is independent of time. Integrating it in time, we get
\begin{equation}\label{inte2}
\frac{1}{Y_2(t')}\geq \frac{1}{Y_2(\bar{x}',0)}+\frac{1}{2}\int_0^{t'}K_c\,ds,
\end{equation}
where the integral is along $\Lambda$. Since $K_c$ is a positive constant, the right hand side of \eqref{inte2} approaches zero in finite time. This implies that $Y_2(t')$ approaches $-\infty$ in finite time.
%
This completes the proof of Theorem \ref{thm_main3}.

\end{proof}

\section*{Acknowledgments}
The work of HC is partially supported by the National Natural Science Foundation of China (No.11801295)  and the Shandong Provincial Natural Science Foundation, China (No.ZR2018BA008). The work of GC is partially supported by  NSF with grant DMS-1715012.
The work of Tian-Yi Wang is partially supported by  NSFC grant 11601401 and 11971024.


\begin{bibdiv}
\begin{biblist}




	\bib{Alinhac1}{article}{
	author={Alinhac, S.},
	title={Temps de vie des solutions r\'{e}guli\`{e}res des \'{e}quations d'Euler compressibles axisym\'{e}triques en dimension deux},
	journal={Invent. Math.},
	volume={11},
	date={1993},
	number={3},
	pages={627--670},
}
\bib{Alinhac-book1}{book}{
	author={Alinhac, S.},
	title={Blowup for nonlinear hyperbolic equations},
	series={Progress in Nonlinear Differential Equations and their
		Applications},
	volume={17},
	publisher={Birkh\"{a}user Boston, Inc., Boston, MA},
	date={1995},
	pages={xiv+113},
}

\bib{Alinhac2}{article}{
	author={Alinhac, S.},
	title={Blowup of small data solutions for a quasilinear wave equation in
		two space dimensions},
	language={English, with English and French summaries},
	journal={Ann. of Math. (2)},
	volume={149},
	date={1999},
	number={1},
	pages={97--127},
}


\bib{Alinhac3}{article}{
	author={Alinhac, S.},
	title={Blowup of small data solutions for a class of quasilinear wave
		equations in two space dimensions. II},
	journal={Acta Math.},
	volume={182},
	date={1999},
	number={1},
	pages={1--23},
	
}


\bib{G3}{article}{
   author={Chen, G.},
   title={Formation of singularity and smooth wave propagation for the
   non-isentropic compressible Euler equations},
   journal={J. Hyperbolic Differ. Equ.},
   volume={8},
   date={2011},
   number={4},
   pages={671--690},
}
\bib{G9}{article}{
   author={Chen, G.},
   title={Optimal time-dependent lower bound on density for classical
   solutions of 1-D compressible Euler equations},
   journal={Indiana Univ. Math. J.},
   volume={66},
   date={2017},
   number={3},
   pages={725--740},
}
\bib{CCZ}{article}{
   author={Chen, G.},
   author={Chen, G.Q.},
   author={Zhu, S.G.},
   title={Formation of singularities and existence of global continuous solutions for the compressible Euler equations},
   journal={preprint, available at arXiv:1905.07758},
}

\bib{CPZ}{article}{
   author={Chen, G.},
   author={Pan, R.H.},
   author={Zhu, S.G.},
   title={Singularity formation for the compressible Euler equations},
   journal={SIAM J. Math. Anal.},
   volume={49},
   date={2017},
   number={4},
   pages={2591--2614},
}


\bib{G5}{article}{
   author={Chen, G.},
   author={Young, R.},
   title={Smooth solutions and singularity formation for the inhomogeneous
   nonlinear wave equation},
   journal={J. Differential Equations},
   volume={252},
   date={2012},
   number={3},
   pages={2580--2595},
}

\bib{G6}{article}{
   author={Chen, G.},
   author={Young, R.},
   title={Shock-free solutions of the compressible Euler equations},
   journal={Arch. Ration. Mech. Anal.},
   volume={217},
   date={2015},
   number={3},
   pages={1265--1293},
}

\bib{G8}{article}{
   author={Chen, G.},
   author={Young, R.},
   author={Zhang, Q.T.},
   title={Shock formation in the compressible Euler equations and related
   systems},
   journal={J. Hyperbolic Differ. Equ.},
   volume={10},
   date={2013},
   number={1},
   pages={149--172},
}

\bib{Chen-1}{article}
{author={Chen, G.-Q.},
	title={ Remarks on spherically symmetric solutions of the compressible Euler equations},
	journal={Proc. Roy. Soc. Edinburgh Sect. A}
	volume={127},
	date={1997}
	number={2}
	page={243--259}, }

\bib{cp1}{article}{
   author={Chen, G.-Q.},
   author={Perepelitsa, M.},
   title={Vanishing viscosity solutions of the compressible Euler equations
   with spherical symmetry and large initial data},
   journal={Comm. Math. Phys.},
   volume={338},
   date={2015},
   number={2},
   pages={771--800},
}



\bib{Chenshuxing}{article}{
   author={Chen, S.X.},
   author={Dong, L.M.},
   title={Formation and construction of shock for p-system},
   journal={Science in China},
   volume={44},
   date={2001},
   number={9},
   pages={1139--1147},
}

\bib{Chenshuxing2}{article}{
   author={Chen, S.X.},
   author={Xin, Z.P.},
   author={Yin, H.C.},
   title={Formation and construction of shock wave for quasilinear hyperbolic system and its application to inviscid compressible flow},
   journal={Research Report, IMS, CUHK},
   date={1999},
}

\bib{Chris}{book}{
   author={Christodoulou, D.},
   title={The shock development problem},
   publisher={arXiv:1705.00828},
   date={2017},
}


\bib{shuang}{book}{
   author={Christodoulou, D.},
   author={Miao, S.},
   title={Compressible flow and Euler's equations},
  series={Surveys of Modern Mathematics},
      volume={9},
   publisher={International Press, Somerville, MA; Higher Education Press, Beijing},
   date={2014},
   pages={x+iv+583},
}

\bib{courant}{book}{
   author={Courant, R.},
   author={Friedrichs, K. O.},
   title={Supersonic Flow and Shock Waves},
   publisher={Interscience Publishers, Inc., New York, N. Y.},
   date={1948},
   pages={xvi+464},
}
\bib{Dafermos2010}{book}{
   author={Dafermos, Constantine M.},
   title={Hyperbolic conservation laws in continuum physics},
   series={Grundlehren der Mathematischen Wissenschaften [Fundamental
   Principles of Mathematical Sciences]},
   volume={325},
   publisher={Springer-Verlag, Berlin},
   date={2000},
   pages={xvi+443},
}
\bib{jenssen}{article}{
   author={Jenssen, Helge K.},
   title={On exact solutions of rarefaction-rarefaction interactions in
   compressible isentropic flow},
   journal={J. Math. Fluid Mech.},
   volume={19},
   date={2017},
   number={4},
   pages={685--708},
   issn={1422-6928},
}
\bib{Fritzjohn}{article}{
   author={John, F.},
   title={Formation of Singularities in One-Dimensional
   Nonlinear Wave Propagation},
   journal={Comm. Pure Appl. Math.},
   volume={27},
   date={1974},
   pages={377--405},
}
\bib{Kongdexing}{article}{
   author={Kong, D.X.},
   title={Formation and propagation of Singularities for $2\times 2$ quasilinear hyperbolic systems},
   journal={Tran.  Amer. Math. Soci.},
   volume={354},
   date={2002},
   pages={3155--3179},
}


\bib{Klainerman}{article}{
	author={Klainerman, S.},
	title={Long time behaviour of solutions to nonlinear wave equations},
	conference={
		title={Nonlinear variational problems},
		address={Isola d'Elba},
		date={1983},
	},
	book={
		series={Res. Notes in Math.},
		volume={127},
		publisher={Pitman, Boston, MA},
	},
	date={1985},
	pages={65--72},,
}

\bib{lfw}{article}{
   author={LeFloch, Philippe G.},
   author={Westdickenberg, M.},
   title={Finite energy solutions to the isentropic Euler equations with
   geometric effects},
   language={English, with English and French summaries},
   journal={J. Math. Pures Appl. (9)},
   volume={88},
   date={2007},
   number={5},
   pages={389--429},
}


\bib{lax2}{article}{
   author={Lax, Peter D.},
   title={Development of singularities of solutions of nonlinear hyperbolic
   partial differential equations},
   journal={J. Mathematical Phys.},
   volume={5},
   date={1964},
   pages={611--613},
}
\bib{Leiduzhang}{article}{
   author={Lei, Z.},
   author={Du, Y.},
   author={Zhang, Q.T.},
   title={Singularities of solutions to compressible Euler equations with
   vacuum},
   journal={Math. Res. Lett.},
   volume={20},
   date={2013},
   number={1},
   pages={41--50},
}

\bib{Li-book1}{book}{
	author={Li, T.(Daqian)},
	author={W.C. Yu},
	title={Boundary value problem for quasilinear hyperbolic systems},
	publisher={Duke University},
	date={1985},
	pages={XXXXXXX},
}

\bib{Li-book2}{book}{
	author={Li, T.(Daqian)},
	title={Global classical solutions for quasilinear hyperbolic system},
	publisher={John Wiley and Sons},
	date={1994},
	pages={XXXXXXX},
}


\bib{Lidaqian}{article}{
   author={Li, T.(Daqian)},
   author={Zhou, Y.},
   author={Kong, De-Xing},
   title={Global classical solutions
   for general quasilinear hyperbolic systems with decay initial data},
   journal={Nonlinear Analysis, Theory, Methods $\&$ Applications},
   volume={28},
   date={1997},
   number={8},
   pages={1299-1332},


}
\bib{lin2}{article}{
   author={Lin, L.W.},
   title={On the vacuum state for the equations of isentropic gas dynamics},
   journal={J. Math. Anal. Appl.},
   volume={121},
   date={1987},
   number={2},
   pages={406--425},
}

\bib{liu}{article}{
   author={Liu, T.P.},
   title={Development of singularities in the nonlinear waves for
   quasilinear hyperbolic partial differential equations},
   journal={J. Differential Equations},
   volume={33},
   date={1979},
   number={1},
   pages={92--111},
}
\bib{ls}{article}{
   author={Liu, T. P.},
   author={Smoller, J. A.},
   title={On the vacuum state for the isentropic gas dynamics equations},
   journal={Adv. in Appl. Math.},
   volume={1},
   date={1980},
   number={4},
   pages={345--359},
}
\bib{mmu}{article}{
   author={Makino, T.},
   author={Mizohata, K.},
   author={Ukai, S.},
   title={The global weak solutions of compressible Euler equation with
   spherical symmetry},
   journal={Japan J. Indust. Appl. Math.},
   volume={9},
   date={1992},
   number={3},
   pages={431--449},
}
\bib{mmu2}{article}{
   author={Makino, T.},
   author={Mizohata, K.},
   author={Ukai, S.},
   title={Global weak solutions of the compressible Euler equation with
   spherical symmetry. II},
   journal={Japan J. Indust. Appl. Math.},
   volume={11},
   date={1994},
   number={3},
   pages={417--426},
}
\bib{mt}{article}{
   author={Makino, T.},
   author={Takeno, S.},
   title={Initial-boundary value problem for the spherically symmetric
   motion of isentropic gas},
   journal={Japan J. Indust. Appl. Math.},
   volume={11},
   date={1994},
   number={1},
   pages={171--183},
}


\bib{Makino}{article}{
   author={Makino,T.},
      author={Ukai,S.},
      author={Kawashima,S. },
         title={Sur la solution $\grave{\text{a}}$ support compact de equations d'Euler compressible},
   journal={Japan. J. Appl. Math.},
   volume={33},
   date={1986},
   pages={249--257},
}

\bib{Speck}{article}{
   author={Luk,J.},
   author={Speck,J.},
   title={Shock formation in solutions to the 2D compressible Euler equations in the presence of non-zero vorticity},
   journal={Inventiones Mathematicae},
   volume={214},
   date={2018},
   pages={1-169},
}
\bib{Riemann}{article}{
   author={Riemann, B.},
   title={Ueber die Fortpflanzung ebener Luftwellen von endlicher Schwingungsweite},
   journal={Abhandlungen der Kniglichen Gesellschaft der Wissenschaften zu Gottingen},
   volume={8},
   date={1860},
   pages={43},
}
\bib{Rammaha}{article}{
   author={Rammaha, Mohammad A.},
   title={Formation of singularities in compressible fluids in two space dimension},
   journal={Proc. Amer. Math. Soc.},
   volume={107},
   date={1989},
   pages={705-714},
}

\bib{Sideris}{article}{
   author={Sideris,Thomas C.},
   title={Formation of singulirities in three-dimensional compressible fluids},
   journal={Comm.Math.Phys.},
   volume={101},
   date={1985},
   pages={475-487},
}

	\bib{Sideris1}{article}{
	author={Sideris,Thomas C.},
	title={Delay singularity formation in 2D compressible flow},
	journal={Amer. J. Math.},
	volume={119},
	date={1997},
	number={2},
	pages={371--422},
}
\bib{smoller}{book}{
   author={Smoller, J.},
   title={Shock waves and reaction-diffusion equations},
   series={Grundlehren der Mathematischen Wissenschaften [Fundamental
   Principles of Mathematical Science]},
   volume={258},
   publisher={Springer-Verlag, New York-Berlin},
   date={1983},
   pages={xxi+581},
}
\bib{stokes}{article}{
   author={Stokes, G.G.},
   title={On a difficulty in the theory of sound},
   journal={Philos. Magazine},
   volume={33},
   date={1848},
   number={3},
   pages={349-356},
}
\bib{TW}{article}{
   author={Tadmor, E.},
   author={Wei, D.M.},
   title={On the global regularity of subcritical Euler-Poisson equations
   with pressure},
   journal={J. Eur. Math. Soc. (JEMS)},
   volume={10},
   date={2008},
   number={3},
   pages={757--769},
}

\end{biblist}
\end{bibdiv}
\end{document}